\documentclass{amsart}
\usepackage{amsmath}
\usepackage{amscd}
\usepackage{amssymb}
\usepackage{amsthm}
\usepackage{enumerate}
\usepackage{comment}
\usepackage[normalem]{ulem}
\usepackage{mathtools}
\usepackage[usenames,dvipsnames]{xcolor}
\usepackage{stmaryrd}
\usepackage[T2A]{fontenc}
\usepackage[utf8]{inputenc}
\usepackage{mathrsfs}
\usepackage{appendix}
\newcounter{pcounter}
\theoremstyle{plain}
\newtheorem{thm}{Theorem}[section]

\newtheorem{prop}[thm]{Proposition}
\newtheorem{lem}[thm]{Lemma}

\newtheorem{cor}[thm]{Corollary}

\theoremstyle{definition}
\newtheorem{defn}[thm]{Definition}

\theoremstyle{remark}

\newtheorem{example}[thm]{Example}

\newcommand{\bC}{{\mathbb C}}

\newcommand{\bE}{{\mathbb E}}
\newcommand{\bF}{{\mathbb F}}

\newcommand{\bN}{{\mathbb N}}

\newcommand{\bZ}{{\mathbb Z}}

\newcommand{\cF}{{\mathcal F}}

\newcommand{\cH}{{\mathcal H}}

\newcommand{\cK}{{\mathcal K}}

\newcommand{\cM}{{\mathcal M}}
\newcommand{\cN}{{\mathcal N}}
\newcommand{\cO}{{\mathcal O}}

\newcommand{\cR}{{\mathcal R}}

\newcommand{\cU}{{\mathcal U}}

\newcommand{\actson}{{\curvearrowright}}

\newcommand\ev{\operatorname{ev}}

\newcommand\dom{\operatorname{dom}}

\newcommand\Hom{\operatorname{Hom}}

\newcommand\orb{\operatorname{orb}}

\newcommand\Proj{\operatorname{Proj}}
\newcommand\Prob{\operatorname{Prob}}

\newcommand\rea{\operatorname{Re}}

\newcommand\Span{\operatorname{span}}

\newcommand\Tr{\operatorname{Tr}}

\newcommand{\ip}[1]{\left\langle#1\right\rangle}
\newcommand{\norm}[1]{\left\lVert#1\right\rVert}

\newcommand{\eps}{\varepsilon}

\title{Consequences of the random matrix solution to the Peterson-Thom conjecture}

\author{Ben Hayes}
\address{\parbox{\linewidth}{Department of Mathematics, University of Virginia, \\
141 Cabell Drive, Kerchof Hall,
P.O. Box 400137
Charlottesville, VA 22904}}
\email{brh5c@virginia.edu}
\urladdr{https://sites.google.com/site/benhayeshomepage/home}

\author{David Jekel}
\address{\parbox{\linewidth}{Universitets parken 5,
2100 Copenhagen $\varnothing$}}
\email{daj@math.ku.dk}
\urladdr{https://davidjekel.com/}

\author{Srivatsav Kunnawalkam Elayavalli}
\address{\parbox{\linewidth}{Department of Mathematics, University of California, \\
San Diego, 9500 Gilman Drive \# 0112, La Jolla, CA 92093}}
\email{srivatsav.kunnawalkam.elayavalli@vanderbilt.edu}
\urladdr{https://sites.google.com/view/srivatsavke}

\thanks{B. Hayes gratefully acknowledges support from the NSF grant DMS-2000105. D. Jekel gratefully acknowledges support from the NSF grant DMS-2002826.  S. Kunnawalkam Elayavalli gratefully acknowledges support from the Simons Postdoctoral Fellowship.}

\begin{document}

\begin{abstract}
We show  various new structural properties of free group factors using the recent resolution (due independently to Belinschi-Capitaine and Bordenave-Collins) of the Peterson-Thom conjecture. These results include  the resolution to the coarseness conjecture independently due to the first-named author and Popa,   a generalization of Ozawa-Popa's celebrated strong solidity result  using vastly more general versions of the normalizer (and in an ultraproduct setting),  
    a dichotomy result for intertwining of maximal amenable subalgebras of interpolated free group factors,  as well as application to ultraproduct embeddings of nonamenable subalgebras of interpolated free group factors.

\end{abstract}
\maketitle

\section{Introduction}


The structure of the group von Neumann algebras associated to the countable free groups (also known as free group factors), has been a constant source of new results and new mysteries.
Murray and von Neumann in \cite{Murray-Neumann36} showed that the free group factors are full, i.e. they have no central sequences, and they used this structural property to distinguish the free group factors from the separable hyperfinite II$_1$  factor, thus giving the first example of two provably non-isomorphic separable II$_1$ factors.   Their work left behind the now notorious open question of whether the free group factors themselves are isomorphic for different numbers of generators.  Almost a century has passed in the development of II$_1$ factors, in which the quest to understand the structure of free group factors has been a recurring theme with several remarkable achievements.

One avenue of this research is the structure of subalgebras of free group factors.
A foundational
discovery of Popa \cite{PopaMaximalAmenable} showed that every subalgebra that strictly contains the generator MASA (maximally abelian subalgebra) in a free group factor must be full, in particular nonamenable (amenability is equivalent to hyperfiniteness by  fundamental work of Connes \cite{Connes}).
This answered in the negative a question of Kadison at the 1967 Baton Rouge conference who asked if every self-adjoint operator in a $\textrm{II}_{1}$-factor is contained in a hyperfinite subfactor.
The technique of asymptotic orthogonality developed by Popa to achieve the above result has been used successfully to establish this maximal  amenability property for various natural subalgebras of the free group factors, such as the radial MASA \cite{CFRW} (see also \cite{2AuthorsOneCup, ParShiWen}). Recently Boutonnet and Popa also construct a continuum size family $(M_{\alpha})_{\alpha}$ of interesting maximally amenable subalgebras \cite{boutonnet2023maximal} in any free product of diffuse tracial von Neumann algebras (in particular for free group factors) with the property that $M_{\alpha}$ is not unitarily conjugate to $M_{\beta}$ if $\alpha\ne\beta$.   

Maximal amenability can be enhanced to an absorption phenomenon as follows. We say that a diffuse $P\leq M$ has the \emph{absorbing amenability property} if whenever $Q\leq M$ is  amenable, and $P\cap Q$ is diffuse, then $Q\leq P$. By modifying Popa's asymptotic orthogonality property, it was shown in \cite{CyrilAOP, WenAOP} respectively that the generator MASA and the radial MASA  admit the absorbing amenability property. This work inspired many papers establishing the absorbing amenability property (and other absorption properties such as Gamma stability) in many examples (see \cite{2AuthorsOneCup, FreePinsker, ParShiWen}).

Given a finite\footnote{ it is not necessary for $M$ to be finite. However, it could be argued that in the general setting one should require all subalgebras to be images of normal conditional expectations. We leave it to those more versed in  Tomita-Takesaki theory to work out the appropriate definition here.} von Neumann algebra $M$, we say that $M$ has \emph{unique maximal amenable extensions} if for every diffuse, amenable subalgebra $Q\leq M$ there is a \emph{unique} maximal amenable $P\leq M$ with $Q\subseteq P$. In \cite{PetersonThom}, Peterson-Thom  conjectured  that any diffuse, amenable subalgebra of a free group factor has unique maximal amenable extensions, which came to be known the \emph{Peterson-Thom conjecture}. This conjecture was motivated by both Peterson-Thom's analogous  insights  on groups with positive first $L^{2}$-Betti numbers and previous work of Ozawa-Popa \cite{OzPopaCartan}, Peterson \cite{PetersonDeriva}, and Jung \cite{JungSB}. One can apply Zorn's lemma to show that for any von Neumann algebra $M$, and any amenable $Q\leq M$, there is \emph{some} maximal amenable $P\leq M$ with $Q\subseteq P$. The novelty of the Peterson-Thom conjecture is that such a $P$ should be \emph{unique}. The Peterson-Thom conjecture is then equivalent to the statement that every maximal amenable subalgebra of a free group factor has the absorbing amenability property.

In \cite{HayesPT}, the first-named author formulated a conjecture on random matrices which he showed to imply the Peterson-Thom conjecture. Several works in random matrices made progress towards this random matrix conjecture \cite{MatrixConc, CGPStrongTen, ParraudAsym}. Recent breakthroughs of Belinschi and Capitaine \cite{PTProperty} and of Bordenave and Collins \cite{bordenave2023norm} independently prove this random matrix conjecture, thus resolving the Peterson-Thom conjecture in the positive.

The reduction of the Peterson-Thom conjecture to a random matrix problem uses Voiculescu's microstates free entropy dimension theory (see \cite{VoicAsyFree, Voiculescu95, FreeEntropyDimensionII, FreeEntropyDimensionIII}), namely the $1$-bounded entropy implicitly defined by Jung \cite[Theorem 3.2]{JungSB} and explicitly by the first-named author in \cite{Hayes8}. We denote the $1$-bounded entropy of an algebra $M$ by $h(M)$ (see Appendix A for the precise definition, which we will not need in the main body of the paper).
The $1$-bounded entropy has several permanence properties which show that the collection of algebras $Q$ with $h(Q:M)\leq 0$ is invariant under various operations such as taking subalgebras, taking the von Neumann algebra generated by its normalizer (or other weakenings of the normalizer) of an algebra, or taking the join of two algebras with diffuse intersection; see Section \ref{sec: 1 bounded entropy} for a list of such properties. Because of the permanence properties that the $1$-bounded entropy enjoys, solving the Peterson-Thom conjecture via $1$-bounded entropy proves several results beyond showing that free group factors have the absorbing amenability property.
This paper will explain in detail several corollaries of the recent resolution of the Peterson-Thom conjecture. As shown in  \cite{PopaWeakInter}, solving the Peterson-Thom property via $1$-bounded entropy also resolves it for the interpolated free group factors $L(\bF_{t})$ independently defined by Dykema \cite{DykemaIFG} and R{\^a}dulescu \cite{RadulaescuIFG}, for $t>1$.
We give a separate proof of this in Section \ref{sec: Pinsker algebra interpolate}.  In fact, because of our work in Section \ref{sec: Pinsker algebra interpolate}, all the main results in this paper apply to interpolated free group factors, and not just free group factors.

Our first main result is the positive resolution of the \emph{coarseness conjecture} independently due to the first-named author \cite[Conjecture 1.12]{Hayes8} and Popa \cite[Conjecture 5.2]{PopaWeakInter}. If $M$ is a von Neumann algebra, an $M$-$M$ bimodule $\cH$ is a Hilbert space with normal left and right actions of $M$ which commute. We use $_M\cH_M$ to mean that $\cH$ is an $M$-$M$ bimodule. If $\cH,\cK$ are $M$-$M$ bimodules, we use $_M\cH_M\leq _M\cK_M$ to mean that there is $M$-bimodular isometry $\cH\to \cK$. If $\cH_{1}\subseteq \cH_{2}$ are Hilbert spaces, we use $\cH_{2}\ominus \cH_{1}$ for $\cH_{1}^{\perp}\cap \cH_{2}$.

\begin{thm}\label{thm:coarseness intro}
Let $t>1$. For any maximal amenable subalgebra $P\leq L(\bF_{t})$ we have
\[_{P}[L^{2}(L(\bF_{t}))\ominus L^{2}(P)]_{P}\leq _{P}(L^{2}(P)\otimes L^{2}(P))^{\oplus \infty} \,_{P}.\]
\end{thm}
In \cite{PopaWeakInter}, this property is referred to as \emph{coarseness} of the inclusion $P\leq L(\bF_{t})$. As explained in the introduction to \cite{PopaWeakInter}, we may think of coarseness as the ``most random'' position a subalgebra can be in.

It is of interest to specialize Theorem \ref{thm:coarseness intro} to the case where $P$ is abelian.
Suppose $(M,\tau)$ is a tracial von Neumann algebra, and $A\leq M$ is a maximal abelian $*$-subalgebra. Write $A=L^{\infty}(X,\mu)$ for some compact Hausdorff space $X$ and some Borel probability measure $\mu$ on $X$. The representation
\[\pi\colon C(X)\otimes C(X)\to B(L^{2}(M)\ominus L^{2}(A))\]
given by
\[\pi(f\otimes g)\xi=f\xi g,\]
gives rise to a spectral measure $E$ on $X\times X$ whose marginals are Radon-Nikodym equivalent to $\mu$. We say that $\nu\in \Prob(X\times X)$ is a \emph{left/right measure} of $A\leq M$ if it is Radon-Nikodym equivalent to $E$. One often abuses terminology and refers to \emph{the} left/right measure to refer to any element of this equivalence class of measures. The measure class of this measure $\nu$, together with the multiplicity function $m\colon X\times X\to \bN\cup\{0\}\cup \{\infty\}$ is usually called the measure-multiplicity invariant (see \cite{FelMoore, NeshStorm}), which has also been investigated in \cite{DykMukLap, MukMasa, PopaWeakInter}. The essential range of the multiplicity function is called the Pukanszky invariant; this was defined in \cite{Puk} and further studied in \cite{DykSinSmiPuk, PopaInterwineSpace, RadulPik,RobStegPuk, SinSmithPUk, WhitePuk}. By \cite[Theorem 4.1]{PopaGeThin} and \cite[Corollary 3.8 (1)]{PopaInterwineSpace} we know that every MASA in an interpolated free group factor has unbounded multiplicity function.

\begin{thm}
Let $M=L(\bF_{t})$ for $t>1$. Suppose that $A\leq M$ is abelian and a maximal amenable subalgebra of $M$. Write $A=L^{\infty}(X,\mu)$ for some compact metrizable space $X$ and some Borel probability measure on $X$. Then the left/right measure of $A\leq M$ is absolutely continuous with respect to $\mu\otimes \mu$.
\end{thm}

Our work shows that for MASAs in interpolated free group factors which are also maximal amenable, the measure given in the measure-multiplicity invariant has to be absolutely continuous with respect to the product measure $\mu\otimes \mu$. More concretely, if we use the fact that all standard probability spaces are isomorphic to reduce to the case where $(X,\mu)$ is $[0,1]$ with Lebesgue measure, then the measure given in the measure-multiplicity invariant has to be absolutely continuous with respect to Lebesgue measure on the unit square.

One of the landmark structural results about free group factors is the solidity property that the commutant of any diffuse subalgebra is amenable. Ozawa achieved this result first by introducing $\mathrm{C}^*$-algebraic boundary techniques \cite{OzawaSolid}.  Popa then gave a different proof using his influential s-malleable deformations and spectral gap rigidity ideas \cite{Popasolidity}.
Peterson verified solidity for more examples using a conceptually new approach based on the theory of closable derivations \cite{PetersonDeriva}. 
All of these results apply to algebras much more general than free group factors, for instance von Neumann algebras of hyperbolic groups (see \cite{ChifanSinclair, OzPopaII, PopaVaesFree, PopaVaesHyp, SinclairGuassian, dkep2022properly}).

Despite the early success of free entropy theory in establishing global structural properties of free group factors (for instance absence of Cartan subalgebras \cite{FreeEntropyDimensionIII} and primeness \cite{GePrime}), solidity was still out of reach by free entropy methods.  In this paper our first result is a proof of Ozawa's solidity theorem based on free entropy dimension techniques, which is completely different from previous arguments.
In fact, for interpolated free group factors, we strengthen the celebrated strong solidity theorem of Ozawa-Popa \cite[Corollary 1]{OzPopaCartan} using a vastly more general version of the normalizer. A first example is the $1$-sided quasi-normalizer defined in \cite{IzumiLongoPopa,PimnserPopa,PopaQR} (building off of ideas in \cite{PopaOrthoPairs}),
\[q^{1}\mathcal{N}_{M}(N)=\left\{x\in M:\mbox{ there exists $x_{1},\cdots,x_{n}\in M$ so that } xN\subseteq \sum_{j=1}^{n}Nx_{j}\right\},\]
we also consider the $wq$-normalizer, defined in \cite{PopaStrongRigidity, PopaCohomologyOE, IPP, GalatanPopa},
\[\mathcal{N}^{wq}_{M}(N)=\{u\in \mathcal{U}(M):uNu^{*}\cap N\textnormal{ is diffuse}\}\]
and its cousin the very weak quasi-normalizer
\[
\mathcal{N}^{vwq}_{M}(N)= \{u\in \mathcal{U}(M): \text{there exists } v\in \mathcal{U}(M) \text{ so that } uNv\cap N\\ \text{ is diffuse.}\}.
\]
Since $uNv\cap N$ is not an algebra, the phrase ``$uNv\cap N$ is diffuse'' should be interpreted as saying that there is a sequence of unitaries $v_{n}\in uNv\cap N$ which tend to zero weakly.
We also consider the \emph{weak intertwining space} $wI_{M}(Q,Q)$ due to Popa \cite{PopaMC, PopaWeakInter} (we restate the definition in Definition \ref{defn:weak intertwining space} of this paper). As shown in \cite[Proposition 3.2]{Hayes8} and Section \ref{sec: singular subspace}, all of these are contained in the anti-coarse space,
 \[
 \cH_{\textnormal{anti-c}}(N\leq M)=\bigcap_{T\in \Hom_{N-N}(L^{2}(M),L^{2}(N)\otimes L^{2}(N))}\ker(T).
\]
Here $\Hom_{N-N}(L^{2}(M),L^{2}(N)\otimes L^{2}(N)$ is the space of bounded, linear, $N$-$N$ bimodular maps $T\colon L^{2}(M)\to L^{2}(N)\otimes L^{2}(N)$. Our next main result is a statement about the most general setting of the anti-coarse space, however for the readers sake we state it in the context of these four examples.

\begin{thm}\label{intro thm: solidity}
Let $t > 1$ and let $Q\leq L(\bF_{t})$ be a diffuse, amenable subalgebra. Then $W^{*}(\cH_{\textnormal{anti-c}}(Q\leq L(\bF_{t})))$ remains amenable. In particular, for any
\[
X\subseteq q^{1}\cN_{L(\bF_{t})}(Q)\cup \cN_{L(\bF_{t})}^{wq}(Q)\cup wI_{L(\bF_{t})}(Q,Q)\cup\mathcal{N}^{vwq}_{L(\bF_{t})}(Q)
\]
we have that $W^*(X)$ is amenable.
\end{thm}

We refer the reader to Section \ref{sec: omega-mega solidity} for the precise definition of $W^{*}(Y)$ for $Y \subseteq L^{2}(M)$. The case $X=\cN_{L(\bF_{t})}(Q)$ in the above theorem recovers the strong solidity theorem of Ozawa-Popa \cite[Corollary 1]{OzPopaCartan} for the special case of interpolated free group factors. Theorem \ref{intro thm: solidity} is already new for $t\in \bN$ for and when  $X\in \{q^{1}\cN_{L(\bF_{t})}(Q),\cN^{wq}_{L(\bF_{t})}(Q), wI_{L(\bF_{t})}(Q,Q),\cN^{vwq}_{L(\bF_{t})}(Q)\}$.
We may, in fact, deduce a further generalization of strong solidity for interpolated free group factors in an ultraproduct framework.

\begin{thm}
Let $t\in (1,+\infty)$ and let $\omega$ be a free ultrafilter on $\bN$.  Suppose that $Q\leq L(\bF_{t})^{\omega}$ is a diffuse, amenable subalgebra. Suppose we are given Neumann subalgebras $Q_{\alpha}$ defined for ordinals $\alpha$ satisfying the following:
\begin{itemize}
    \item $Q_{0}=Q$,
    \item if $\alpha$ is a successor ordinal then $Q_{\alpha}=W^{*}(X_{\alpha},Q_{\alpha-1})$ where
    \[X_{\alpha}\subseteq \cH_{\textnormal{anti-c}}(Q_{\alpha-1}\leq L(\bF_{t})),\]
    (for example if
    \[X_{\alpha}\subseteq q^{1}\cN_{L(\bF_{t})^{\omega}}(Q_{\alpha-1})\cup \cN_{L(\bF_{t})^{\omega}}^{wq}(Q_{\alpha-1})\cup wI_{L(\bF_{t})^{\omega}}(Q_{\alpha-1},Q_{\alpha-1})\cup \cN^{vwq}_{M}(Q_{\alpha-1}) )\]
    \item if $\alpha$ is a limit ordinal, then $Q_{\alpha}=\overline{\bigcup_{\beta<\alpha}Q_{\beta}}^{SOT}$.
\end{itemize}
Then for any ordinal $\alpha$ we have $Q_{\alpha}\cap L(\bF_{t})$ is amenable.
In particular, $L(\bF_{t})$ has the following Gamma stability property: if $Q\leq L(\bF_{t})^{\omega}$,
and if $Q'\cap L(\bF_{t})^{\omega}$ is diffuse, then $Q\cap L(\bF_{t})$ is amenable.
\end{thm}
This  recovers the previous Gamma stability results from \cite{CyrilAOP} (recovering also instances of the results of \cite{dkep2022properly, ding2023structure}).

The case of the weak intertwining space itself leads to a dichotomy in terms of Popa's deformation/ridigity theory for maximal amenable subalgebras of free group factors. We recall the fundamental notion of intertwining introduced in \cite{PopaL2Betti, PopaStrongRigidity}.  If $M$ is a finite von Neumann algebra, and $P,Q\leq M$ we say that \emph{a corner of $Q$ intertwines into $P$ inside of $M$} and write $Q\prec P$ if there are nonzero projections $f\in Q$, $e\in P$, a unital $*$-homomorphism $\Theta\colon fQf\to ePe$ and a nonzero partial isometry $v\in M$ so that:
\begin{itemize}
    \item $xv=v\Theta(x)$ for all $x\in fQf$,
    \item $vv^{*}\in (fQf)'\cap fMf$,
    \item $v^{*}v\in \Theta(fqf)'\cap eMe$.
\end{itemize}
This can be thought of intuitively as ``$Q$ can be unitarily embedded into $P$, after cutting by a projection.''

\begin{thm}\label{thm:Pinkser dichotomy intro}
Fix $t>1$, and let $Q,P$ be maximal amenable subalgebras of $L(\bF_{t})$. Then exactly one of the following occurs:
\begin{enumerate}
    \item either there are nonzero projections $e\in Q,f\in P$ and a unitary $u\in L(\bF_{t})$ so that $u^{*}(ePe)u=fQf$, or \label{item:option 1 intro}
    \item for any diffuse $Q_{0}\leq Q$ we have that $Q_{0}\nprec P$.
    \label{item:option 2 intro}
\end{enumerate}
In particular, if $Q,P$ are hyperfinite subfactors of $L(\bF_{t})$ that are maximal amenable subalgebras in $L(\bF_{t})$, then either they are unitarily conjugate or no corner of any diffuse subalgebra of one can be intertwined into the other inside of $L(\bF_{t})$.
\end{thm}

So given any pair $P,Q$ of maximal amenable subalgebras of $L(\bF_{t})$, they either have unitarily conjugate corners or no diffuse subalgebra of $Q$ can be ``essentially conjugated'' into $P$. The reader should compare this with
\cite[Theorem A.1]{PopaL2Betti} where a similar result is shown for MASAs in \emph{any} tracial von Neumann algebra.

We close with an application to embeddings into matrix ultraproducts.

\begin{cor}
Let $t > 1$ and let $N \leq L(\bF_{t})$ be a nonamenable subfactor. Then there is a free ultrafilter $\omega$ and an embedding $\iota\colon N\to \prod_{k\to\omega}M_{k}(\bC)$ with $\iota(N)'\cap \prod_{k\to\omega}M_{k}(\bC)=\bC 1$.
\end{cor}

The corollary is proved as follows.  The results of \cite{HayesPT} and \cite{PTProperty, bordenave2023norm} imply that every nonamenable $N \leq L(\mathbb{F}_t)$ satisfies $h(N: L(\mathbb{F}_t)) > 0$ and hence $h(N) > 0$; see Corollary \ref{Pinsker algebras of interpolated fgf} for details.  Work of the second-named author shows that if $h(N) > 0$, then there exists some embedding of $N$ into a matrix ultraproduct which has trivial relative commutant \cite[Corollary 1.3, see the statement in corrigendum]{Jekeltypes}.





\textbf{A comment on proofs.} We give a few remarks on how $1$-bounded entropy is used in the paper. The first is that for any tracial von Neumann algebra, the $1$-bounded entropy leads to a natural class of subalgebras called \emph{Pinsker algebras}. To say that $P\leq M$ is Pinsker means that $P$ is a maximal subalgebra such that $h(P:M)\leq 0$, where $h(Q:N)$ for $Q\leq N$ is the  $1$-bounded entropy of $Q$ in the presence of $N$ defined in \cite{Hayes8}. Intuitively this means two things: firstly that $P$ has ``very few'' embeddings into ultraproducts of matrices which extend to $M$,  and that $P$ is maximal with respect to inclusion among algebras which have ``very few'' embeddings into ultraproducts of matrices which extend to $M$. We refer to Section
\ref{sec:Pinsker algebra}
for the precise definition of being a Pinsker algebra. Given a diffuse subalgebra $Q\leq M$ with $h(Q:M)\leq 0$, general properties of $1$-bounded entropy show that there is a unique Pinsker algebra $P\leq M$ with $Q\subseteq P$.  In the context of the Peterson-Thom conjecture, the unique maximal amenable extension of a diffuse subalgebra can be identified exactly as the Pinsker algebra containing this subalgebra. Thus the $1$-bounded entropy can not only be used to solve the Peterson-Thom conjecture, but also naturally identify the maximal amenable extensions of a given amenable subalgebra.

The second remark is that, as was previously mentioned, the $1$-bounded entropy enjoys several permanence properties which we will list in Section \ref{sec: 1 bounded entropy}. For the proofs of all the theorems mentioned so far, we will only use these permanence properties as well as the fact that the results of \cite{HayesPT, PTProperty, bordenave2023norm} show that the Pinsker algebras in interpolated free group factors are precisely the maximal amenable subalgebras. In particular, we never have to work with the precise definition of $1$-bounded entropy, just these permanence properties. This tells us that the axiomatic treatment of $1$-bounded entropy via these general properties can be used to deduce many interesting and new results on the structure of free group factors.

\textbf{Organization of the paper.} In Section \ref{sec: singular subspace} we recall the anti-coarse space and expand the results in \cite[Proposition 3.2]{Hayes8} showing that it contains several other weakenings of the normalizer. In Section \ref{sec: 1 bounded entropy}, we list the  permanence properties of $1$-bounded entropy we will use in this paper. For the proofs of all of our applications of \cite{HayesPT} and \cite{PTProperty, bordenave2023norm} we will only use these permanence properties and not the precise definition of $1$-bounded entropy, so these properties give an axiomatic approach to most of the proofs in this paper. In Section \ref{sec:Pinsker algebra}, we recall the notion of Pinsker algebras for $1$-bounded entropy introduced in \cite{FreePinsker} and recast the results of \cite{PTProperty, bordenave2023norm} in these terms. In Section  \ref{sec: Pinsker algebra interpolate}, we will explain how these results apply not just to maximal amenable subalgebras of free group factors, but also to those of interpolated free group factors. Section \ref{sec:orthocomp}
contains the proof of the coarseness conjecture, as well as applications to maximal abelian subalgebras of free group factors. In Section \ref{sec: omega-mega solidity} we give several generalizations of Ozawa-Popa's celebrated strong solidity theorem.  In Section \ref{sec:intertwine} we give a dichotomy for intertwining between maximal amenable subalgebras of interpolated free group factors (and more generally for Pinsker algebras in any tracial von Neumann algebra).
In the appendix, we give the definition of $1$-bounded entropy and prove that it is independent of the choice of generators (a fact first proved implicitly in \cite[Theorem 3.2]{JungSB} and later explicitly in \cite[Theorem A.9]{Hayes8}). Here we give a significant conceptual and technical simplification of the proof using the noncommutative functional calculus due to the second-named author.

 \textbf{Acknowledgements.} We thank Enes Kurt, Sorin Popa, Thomas Sinclair, and Stefaan Vaes for providing many helpful comments that improved the exposition. 



\section{Preliminaries}\label{back}

\subsection{The anti-coarse space}\label{sec: singular subspace}

For an inclusion $N\leq M$ of tracial von Neumann algebras, we let
\[\cH_{\textnormal{anti-c}}(N\leq M)=\bigcap_{T\in \Hom_{N-N}(L^{2}(M),L^{2}(N)\otimes L^{2}(N))}\ker(T).\]
Where $\Hom_{N-N}(L^{2}(M),L^{2}(N)\otimes L^{2}(N)$ is the space of bounded, linear, $N$-$N$ bimodular maps $T\colon L^{2}(M)\to L^{2}(N)\otimes L^{2}(N)$. It is shown in \cite[Proposition 3.2]{Hayes8} that this contains the following generalizations of the normalizer of $N\leq M$
\[q^{1}\mathcal{N}_{M}(N)=\left\{x\in M:\mbox{ there exists $x_{1},\cdots,x_{n}\in M$ so that } xN\subseteq \sum_{j=1}^{n}Nx_{j}\right\},\]
\[\mathcal{N}^{wq}_{M}(N)=\{u\in \mathcal{U}(M):uNu^{*}\cap N\textnormal{ is diffuse}\}.\]

We recall the following definition, due to Popa \cite{PopaMC,PopaWeakInter} (see also \cite{PopaCohomologyOE,IPP, GalatanPopa} for related concepts).
\begin{defn}[Definition 2.6.1 of \cite{PopaWeakInter}] \label{defn:weak intertwining space}
Let $(M,\tau)$ be a tracial von Neumann algebra. For $Q,P\leq M$ diffuse, we define the \emph{intertwining space from $Q$ to $P$ inside $M$}, denoted $I_{M}(Q,P)$, to be the set of $\xi\in L^{2}(M)$ so that
\[\overline{\Span\{a\xi b: a\in Q,b\in P\}}^{\|\cdot\|_{2}}\]
has finite dimension as a right $P$-module. We define the \emph{weak intertwining space from $Q$ to $P$ inside $M$} by
\[wI_{M}(Q,P)=\bigcup_{Q_{0}\leq Q \textnormal{ diffuse}}I_{M}(Q_{0},P).\]

\end{defn}

The following is a well-known result due to Popa \cite[Proposition 2.6.3]{PopaWeakInter}, but we include the proof for completeness.

\begin{prop}\label{prop:singular subspaces contain weak intertwiners}
Let $(M,\tau)$ be a tracial von Neumann algebra. For $Q\leq M$ diffuse we have
\[wI_{M}(Q,Q)\subseteq \cH_{\textnormal{anti-c}}(Q\leq M).\]
\end{prop}

\begin{proof}
Fix $Q_{0}\leq Q$ diffuse. It suffices to show that
\[(I_{M}(Q_{0},Q))^{\perp}\supseteq \cH_{\textnormal{anti-c}}(Q\leq M)^{\perp}.\]
It is a folklore result that $\cH_{\textnormal{anti-c}}(Q\leq M)^{\perp}$ can be embedded into an infinite direct sum of $L^{2}(Q)\otimes L^{2}(Q)$ as a $Q$-$Q$ bimodule (see e.g. \cite[Proposition 3.3]{Hayes8} for a complete proof). Since $Q_{0}$ is diffuse, we can find a sequence $u_{n}\in \mathcal{U}(Q_{0})$ which tend to zero weakly. We leave it as an exercise to check that for all $\xi,\eta\in L^{2}(Q)\otimes L^{2}(Q)$ we have
\[\lim_{n\to\infty}\sup_{y\in Q:\|y\|\leq 1 }|\ip{u_{n}\xi y,\eta}|=0=\lim_{n\to\infty}\sup_{y\in Q:\|y\|\leq 1 }|\ip{y\xi u_{n},\eta}|,\]
(to prove this one can, for instance, check it on simple tensors and then conclude using linearity and density). Since $\cH_{\textnormal{anti-c}}(Q\leq M)^{\perp}$ embeds into an infinite direct sum of $L^{2}(Q)\otimes L^{2}(Q)$ as a $Q$-$Q$ bimodule, it follows that for all $\xi,\eta\in \cH_{\textnormal{anti-c}}(Q\leq M)^{\perp}$ we have
\[\lim_{n\to\infty}\sup_{y\in Q:\|y\|\leq 1 }|\ip{u_{n}\xi y,\eta}|=0=\lim_{n\to\infty}\sup_{y\in Q:\|y\|\leq 1 }|\ip{y\xi u_{n},\eta}|.\]
Since
\[\sup_{y\in Q:\|y\|\leq 1 }|\ip{u_{n}\xi y,\eta}|=\|\bE_{Q}(\eta^{*}u_{n}\xi)\|_{1},\]
we have
\[0=\lim_{n\to\infty}\|\bE_{Q}(\eta^{*}u_{n}\xi)\|_{1}.\]
So \cite[Theorem 1.3.2]{PopaInterwineSpace} implies that $\cH_{\textnormal{anti-c}}(Q\leq M)^{\perp} \subseteq I_{M}(Q_{0},Q)^{\perp}$, as desired.
\end{proof}

To further illustrate the generality of the anti-coarse space, we show that it contains the following very weak normalizer of $(N\subseteq M)$:
\[\mathcal{N}^{vwq}_{M}(N)= \{u\in \mathcal{U}(M): \text{there exists } v\in \mathcal{U}(M) \text{ so that } uNv\cap N\\ \text{ is diffuse.}\}\]
Here by ``diffuse'' we mean that $uNv\cap N$ contains a sequence $(u_n)_{n \in \mathbb{N}}$ of unitaries with $u_n \to 0$ in WOT as $n \to \infty$.

\begin{prop}
Let $(M,\tau)$ be a tracial von Neumann algebra. For $N\leq M$ diffuse we have
\[\mathcal{N}^{vwq}_{M}(N)\subseteq \cH_{\textnormal{anti-c}}(N\leq M).\]
\end{prop}

\begin{proof}
The argument proceeds exactly as in \cite[proof of Proposition 3.2]{Hayes8} but we include it for reader's convenience. Let $u\in \mathcal{N}^{wqv}_{M}(N)$.
By definition, it suffices to show that for every $T\in \Hom_{N-N}(L^{2}(M),L^{2}(N)\otimes L^{2}(N))$ we have $T(u)=0$.
Let $v\in \mathcal{U}(M)$ and $u_n\in \mathcal{U}(N)\cap uNv$
be such that $u_{n}\to_{n\to\infty}^{WOT}0$.
Denote by $w_n=u^{*}u_nv^{*}\in \mathcal{U}(N)$ and observe that $w_n\to 0 $ in WOT.
Then, using that $T$ is $N$-$N$ bimodular and that $u_n,w_n\in \mathcal{U}(N)$,
\begin{align*}
  \| T(u) \|_2^2= \|T(u)w_n \|_2^2=\langle T(u)w_n,T(u)w_n \rangle=\ip{T(u)w_n,T(uw_n)}&=\ip{T(u)w_n,T(u_n v^{*})}\\
  &=\langle u_n^{*}T(u)w_n,T(v^{*}) \rangle.
  \end{align*}
Since $u_{n}^{*}\to_{n\to\infty}^{WOT}0$ and $w_{n}\to_{n\to\infty}^{WOT}0$, it follows as in Proposition \ref{prop:singular subspaces contain weak intertwiners} that $\ip{u_n^{*}\xi w_n,\eta}\to_{n\to\infty}0$ for all $\xi,\eta\in L^{2}(N)\otimes L^{2}(N)$. Taking limits as $n \to\infty$ above thus shows that $T(u)=0$.
\end{proof}

\subsection{1-bounded entropy}\label{sec: 1 bounded entropy}

One of the main ideas going into the proof of the Peterson-Thom conjecture is the $1$-bounded entropy $h$ of a tracial von Neumann algebra, a numerical invariant which appeared implicitly Jung \cite{JungSB} and was made explicit by the first-named author in \cite{Hayes8}. We will need the more general notion of $1$-bounded entropy in the presence, which is defined for inclusions $N\leq M$ of tracial von Neumann algebras and is denoted $h(N:M)$.
For detailed expositions and recent work on this topic see \cite{FreePinsker, PropTS1B, L2BHJKE, Jekeltypes, elayavalli2023remarks, exoticCIKE, charlesworth2023strong}.

For the applications in this paper, we will not need to use the definition of $h$ directly, only the properties listed below.  We include the precise definition of $h$ in the appendix, along with a streamlined proof that the definition is independent of the choice of the generating sets for the von Neumann algebras. The name ``$1$-bounded entropy" derives from the following result, connecting the $1$-bounded entropy to Jung's strong $1$-boundedness \cite{JungSB}.

\begin{thm}[{See  \cite[Proposition A.16]{Hayes8}}] \label{thm:S1B in terms of 1-bounded entropy}
A tracial von Neumann algebra $M$ is strongly $1$-bounded in the sense of Jung \cite{JungSB} if and only if $M$ is finitely generated and diffuse and satisfies $h(M) < \infty$.
\end{thm}

For this reason, we say that any tracial von Neumann algebra $(M,\tau)$ (even if $M$ is not finitely generated or diffuse) is \emph{strongly $1$-bounded} if $h(M)<+\infty$.

Let $(M,\tau)$ be a tracial von Neumann algebra. The $1$-bounded entropy in the presence enjoys the following properties.

\begin{list}{{\bf P\arabic{pcounter}:} ~ }{\usecounter{pcounter}}
\item $h(M)=h(M:M)$ for every  tracial von Neumann algebra $(M,\tau)$, \label{item:definition}
\item Suppose $N\leq M$. Then $h(N:M)\geq 0$ if  $M$ embeds into an ultrapower of $\mathcal{R}$, and $h(N:M)=-\infty$ if $M$ does not embed into an ultrapower of $\mathcal{R}$. (Exercise).
\item $h(N_{1}:M_{1})\leq h(N_{2}:M_{2})$ if $N_{1}\leq N_{2}\leq M_{2}\leq M_{1}$,  \label{I:monotonicity of 1 bounded entropy}(Exercise).
\item $h(N:M)\leq 0$ if $N\leq M$ and $N$ is  hyperfinite.  \label{I:hyperfinite has 1-bdd ent zero} (Exercise).
\item $h(M)=\infty$ if $M$ is diffuse, and $M = \mathrm{W}^*(x_{1},\cdots,x_{n})$ where $x_{j}\in M_{sa}$ for all $1\leq j\leq n$ and $\delta_{0}(x_{1},\cdots,x_{n})>1$. For example, this applies to $M=L(\bF_{n})$ for $n>1$ because of \cite{FreeEntropyDimensionII, FreeEntropyDimensionIII} together with Theorem \ref{thm:S1B in terms of 1-bounded entropy} and \cite[Corollary 3.5]{JungSB}.
\item $h(N_{1}\vee N_{2}:M)\leq h(N_{1}:M)+h(N_{2}:M)$ if $N_{1},N_{2}\leq M$ and $N_{1}\cap N_{2}$ is diffuse. (See \cite[Lemma A.12]{Hayes8}.)\label{I:subadditivity of 1 bdd ent}
\item Suppose that $(N_{\alpha})_{\alpha}$ is an increasing chain of diffuse von Neumann subalgebras of a von Neumann algebra $M$. Then
\[h\left(\bigvee_{\alpha}N_{\alpha}:M\right)=\sup_{\alpha}h(N_{\alpha}:M).\]
\label{I:increasing limits of 1bdd ent first variable} (See \cite[Lemma A.10]{Hayes8}.)
\item $h(N:M)=h(N:M^{\omega})$ if  $\omega$ is a free ultrafilter on an infinite set. (See \cite[Proposition 4.5]{Hayes8}.) \label{I:omegafying in the second variable}
\item $h(W^{*}(\cH_{\textnormal{anti-c}}(N\leq M)):M)=h(N:M)$ if $N\leq M$ is diffuse.   (See \cite[Theorem 3.8]{Hayes8}. See also Section \ref{sec: omega-mega solidity} for the definition of $W^{*}(Y)$ for $Y\subseteq L^{2}(M)$.) \label{item:preservation under normalizers}
\item \label{item:direct sum inequality} Let $I$ be a countable set, and $M=\bigoplus_{i\in I}M_{i}$ with $M_{i}$ diffuse for all $i$. Suppose that $\tau$ is a faithful trace on $M$, and that $\lambda_{i}$ is the trace of the identity on $M_{i}$. Endow $M_{i}$ with the trace $\tau_{i}=\frac{\tau|_{M_{i}}}{\lambda_{i}}$. Fix $N_{i}\leq M_{i}$ for all $i\in I$. Then
\[h(N:M,\tau)\leq \sum_{i}\lambda_{i}^{2}h(N_{i}:M_{i},\tau_{i}).\]
(See the proof of \cite[Proposition A.13 (i)]{Hayes8}.)
\item
\label{item:corner preserving vanishing}
If $z\in \Proj(Z(M))$, $N\leq M$ and $h(N:M)\leq 0$, then $h(Nz:Mz)\leq 0$.
(see \cite[Lemma 4.2]{PropTS1B}).

\item $h(pNp:pMp)=\frac{1}{\tau(p)^{2}}h(N:M)$, if $N\leq M$ is diffuse, $p$ is a nonzero projection in $N$, and $M$ is a factor.
(this follows from modifying the proofs of \cite[Proposition A.13 (ii)]{Hayes8} and \cite[Proposition 4.6]{PropTS1B}).

\label{item: corner inequality}

\end{list}


\subsection{Pinsker algebras}\label{sec:Pinsker algebra}

\begin{defn}
Let $(M,\tau)$ be a tracial von Neumann algebra. We say that $P\leq M$ is \emph{Pinsker} if $h(P:M)\leq 0$ and for any $P\leq Q\leq M$ with $P\ne Q$ we have $h(Q:M)>0$.
\end{defn}

By Properties P\ref{I:subadditivity of 1 bdd ent} and P\ref{I:increasing limits of 1bdd ent first variable}, if $Q\leq M$ is diffuse and $h(Q:M)\leq 0$, then there is a unique Pinsker algebra $P\leq M$ with $Q\subseteq P$. E.g.
\[P=\bigvee_{N\leq M, N\supseteq Q, h(N:M)\leq 0}N.\]
We call $P$ the \emph{Pinsker algebra of $Q\subseteq M$}.
By \cite{HayesPT} and the recent breakthrough result of Belinschi-Capitaine \cite{PTProperty} and Bordenave-Collins \cite{bordenave2023norm}, we have the following classification of Pinsker subalgebras of free group factors.

\begin{thm}\label{thm:classification of Pinsker algebras}
Fix $r\in \bN \cup\{\infty\}$. Then:
\begin{enumerate}[(i)]
\item $Q\leq L(\bF_{r})$ is amenable if and only if $h(Q:L(\bF_{r}))=0$,
\item $P\leq L(\bF_{r})$ is Pinsker if and only if it is maximal amenable.
\end{enumerate}
\end{thm}

As remarked in \cite{FreePinsker}, we may think of $1$-bounded entropy as analogous to the Kolmogorov-Sina\v\i\ entropy in the context of probability measure-preserving actions of groups. Entropy for probability measure-preserving actions of groups was first developed in the case of $\bZ$ by \cite{Kol58, Sin59}, for amenable groups in \cite{Kieff, OrnWeiss}, and then for sofic groups in \cite{Bow,KLi}.
See also \cite{SewardKrieger} for a potential approach to entropy for arbitrary acting groups, called Rokhlin entropy.
A probability measure-preserving action $G\actson (X,\mu)$ with $G$ sofic is said to have \emph{complete positive entropy} if every nontrivial quotient probability measure-preserving quotient action has positive entropy. Dually, this is equivalent to saying that if $B$ is a $G$-invariant von Neumann subalgebra $B$ of $L^{\infty}(X,\mu)$ with $B\ne \bC 1$, then the action of $G$ on $B$ has positive entropy.

Motivated by the sofic entropy case, one could na{\"\i}vely define a tracial von Neumann algebra $(M,\tau)$ to have completely positive $1$-bounded entropy if any nontrivial subalgebra has positive $1$-bounded entropy. However, this will never be satisfied, as any hyperfinite subalgebra will have vanishing $1$-bounded entropy. Thus any tracial von Neumann algebra will have many subalgebras with vanishing $1$-bounded entropy (these subalgebras can be chosen to be diffuse if $M$ is). Instead, we should think of a result saying that the only subalgebras with vanishing $1$-bounded entropy  are ones that can be quickly deduced to  vanishing have $1$-bounded entropy from the Properties P\ref{item:definition}-P\ref{item: corner inequality} listed above (e.g. hyperfinite algebras, property Gamma algebras, nonprime algebras, algebras with a Cartan, etc.) as a complete positive entropy result. We may thus think of Theorem \ref{thm:classification of Pinsker algebras} as a complete positive entropy result for $1$-bounded entropy. Since free group factors may be thought of as the free probability analogue of Bernoulli shifts (e.g. because $L(\bF_{\infty})$ is the crossed product algebra associated to a free Bernoulli shift), Theorem \ref{thm:classification of Pinsker algebras} should be compared with previous results establishing complete positive entropy of Bernoulli shifts (see \cite{RudolphMixing, KerrCPE}).

As discussed in \cite[Section 5]{Pineapple}, Pinsker algebras of measure-preserving dynamical systems are analogous to the maximal rigid subalgebras of $s$-malleable deformations in \cite{Pineapple}. This allows for an exchange of ideas and methods between deformation/rigidity theory, free probability theory, and ergodic theory. This will be exploited in Section \ref{sec:intertwine} where we adapt arguments in \cite{Pineapple} to show that Pinsker algebras do not have any weak intertwiners between them unless they have corners which are unitarily conjugate.

\section{Pinsker algebras of interpolated free group factors}\label{sec: Pinsker algebra interpolate}

As mentioned before, the combined results of \cite{HayesPT} and \cite{PTProperty,bordenave2023norm} prove that for $r\in \bN$ we have that $Q\leq L(\bF_{r})$ is amenable if and only if $h(Q:L(\bF_{r})) = 0$. The main goal of this section is to explain how this automatically generalizes to interpolated free group factors.

\begin{cor}\label{Pinsker algebras of interpolated fgf}
Fix $t>1$. Then $Q\leq L(\bF_{t})$ is amenable if and only if $h(Q:L(\bF_{t})) =0$.
\end{cor}

By rephrasing the above Corollary we obtain:

\begin{cor}\label{cor:classification of Pinsker}
Fix $t>1$. Let $P\leq L(\bF_{t})$. Then $P\leq L(\bF_{t})$ is Pinsker if and only if $P$ is maximal amenable.
\end{cor}

In order to obtain this result, we will study the relationship between Pinsker algebras and compression.  First, we generalize Property P\ref{item: corner inequality} to the case where $M$ is not a factor.

\begin{lem}\label{lem: cut down of 0 is 0}
Let $(M,\tau)$ be a tracial von Neumann algebra and $P\leq M$. If $h(P:M)\leq 0$, then for every projection $p\in P$ we have that $h(pPp:pMp)\leq 0$.
\end{lem}

\begin{proof}
By decomposing the center of $M$ into atomic and diffuse pieces, we can find a central projection $z_{0}\in M$ (potentially zero), a countable set $I$ (potentially empty), and central projections $(z_{i})_{i\in I}$ so that:
\begin{itemize}
    \item $1=z_{0}+\sum_{i}z_{i}$,
    \item in the case $z_{0}\ne 0$ we have that $Mz_{0}$ has diffuse center,
    \item $Mz_{i}$ is a factor for all $i\in I$.
\end{itemize}
For $i\in \{0\}\cup I$, let $P_{i}=(pPp)z_{i}$ (even though $z_{i}$ may not be in $P$, we still have that $P_{i}$ is a von Neumann subalgebra of $M$ as $z_{i}$ is central). Set $\widehat{P}=\overline{\sum_{i}P_{i}}^{WOT}$.
Then, by P\ref{I:monotonicity of 1 bounded entropy} and P\ref{item:direct sum inequality}:
\begin{align}\label{eqn: same trick recycled}
h(pPp:pMp)\leq h(\widehat{P}:pMp)&\leq \tau(pz_{0})^{2}h(P_{0}:(pMp)z_{0})+\sum_{i}\tau(pz_{i})^{2}h(P_{i}:(pMp)z_{i})\\
&\leq \tau(pz_{0})^{2}h((pMp)z_{0})+\sum_{i}\tau(pz_{i})^{2}h(P_{i}:(pMp)z_{i}) \nonumber
\end{align}
So  it suffices to show each term on the right-side of this inequality is nonpositive.

We first show that $\tau(pz_{0})^{2} h((pMp)z_{0})\leq 0$. If $pz_{0}=0$, the claim is true. Otherwise, since $M$ has diffuse center and $Z(pMp)=pZ(M)p)$, we have that $(pMp)z_{0}$ has diffuse center. Thus $(pMp)z_{0}=W^{*}(\cN_{(pMp)z_{0}}(pZ(M)pz_{0}))$, and so  the combination of Properties P\ref{I:hyperfinite has 1-bdd ent zero} and P\ref{item:preservation under normalizers} imply that $h((pMp)z_{0})\leq 0$.

Now consider $h(P_{i}:(pMp)z_{i})$ for $i\in I$. By Property P\ref{item:corner preserving vanishing}, we know that
\[h(Pz_{i}:Mz_{i})\leq 0\]
for all $i\in I$. Thus, by Property P\ref{item: corner inequality} we have that
\[h(P_{i}:(pMp)z_{i})=\frac{1}{\tau(pz_{i})^{2}}h(Pz_{i}:Mz_{i})\leq 0.\]

Thus we have shown that all terms on the right-hand side of (\ref{eqn: same trick recycled}) are nonpositive, and this completes the proof.
\end{proof}

We now show that being a Pinsker algebra is preserved under taking corners and amplifications.

\begin{prop}\label{prop: amplifications of Pinsker}
Let $(M,\tau)$ be a tracial von Neumann algebra, and suppose that $P\leq M$ is Pinsker. Then:
\begin{enumerate}[(i)]
\item we have $Z(M)\subseteq P$.\label{item: cutting to summands}
\item for any nonzero projection $p\in P$ we have that $pPp$ is a Pinsker subalgebra of $pMp$. \label{item: compressions of Pinsker}
\item for any $n\in \bN$, we have that $M_{n}(P)$ is a Pinsker subalgebra of $M_{n}(M)$. \label{item: amplification of Pinkser}
\end{enumerate}
\end{prop}

\begin{proof}

(\ref{item: cutting to summands}):
Note that $Z(M)\vee P\subseteq W^{*}(\cN_{M}(P))$,
and thus, by P\ref{item:preservation under normalizers},
\[h(Z(M)\vee P:M)\leq h(P:M)=0,\]
and so $P$ being Pinsker forces $Z(M)\vee P\subseteq P$. That is, $Z(M)\subseteq P$.

(\ref{item: compressions of Pinsker}): Let $z$ be the central support of $p\in M$. Then, by (\ref{item: cutting to summands}), we know that $z$ is under the central support of $p$ in $P$. So  there exists a collection $\{v_{i}\}_{i\in I}$ of nonzero partial isometries in $P$ so that $v_{i}^{*}v_{i}\leq p$, and $z=\sum_{i}v_{i}v_{i}^{*}$. Set $p_{i}=v_{i}^{*}v_{i}$. We may, and will, assume that there is some $i_{0}$ so that $v_{i_{0}}=p$.  By Lemma \ref{lem: cut down of 0 is 0}, $h(pPp:pMp) \leq 0$ and so there exists a Pinsker subalgebra $Q$ of $pMp$ containing $pPp$.
Set
\[\widehat{Q}=\overline{Q+\sum_{i\in I\setminus \{i_{0}\}}v_{i}v_{i}^{*}Pv_{i}v_{i}^{*}}^{WOT}.\]
Thus, by P\ref{item:direct sum inequality},
\begin{align*}
h(Q:M)&\leq h\left(\widehat{Q}:\overline{pMp+\sum_{i\in I\setminus\{i_{0}\}}v_{i}v_{i}^{*}Mv_{i}v_{i}^{*}}^{WOT}\right)&\\
&\leq \tau(p)^{2}h(Q:pMp)+\sum_{i\in I\setminus\{i_{0}\}}\tau(v_{i}v_{i}^{*})^{2}h(v_{i}v_{i}^{*}Pv_{i}v_{i}^{*}:v_{i}v_{i}^{*}Mv_{i}v_{i}^{*})\\
&\leq \sum_{i\in I\setminus\{i_{0}\}}\tau(v_{i}v_{i}^{*})^{2}h(v_{i}v_{i}^{*}Pv_{i}v_{i}^{*}:v_{i}v_{i}^{*}Mv_{i}v_{i}^{*}).
\end{align*}
For all $i\in I$, we have that $x\mapsto v_{i}xv_{i}^{*}$ gives a trace-preserving isomorphism from $p_{i}Mp_{i}\to v_{i}v_{i}^{*}Mv_{i}v_{i}^{*}$ which takes $p_{i}Pp_{i}$ to $v_{i}v_{i}^{*}Pv_{i}v_{i}^{*}$. Hence for all
$i\in I$
\[h(v_{i}v_{i}^{*}Pv_{i}v_{i}^{*}:v_{i}v_{i}^{*}Mv_{i}v_{i}^{*})=h(p_{i}Pp_{i}:p_{i}Mp_{i})\leq 0,\]
by Lemma \ref{lem: cut down of 0 is 0}. Altogether we have shown that $h(\widehat{Q}:M)\leq 0$. Note that 
\[\widehat{Q}\cap P\supseteq \overline{pPp+\sum_{i\in I\setminus\{i_{0}\}}v_{i}v_{i}^{*}Pv_{i}v_{i}^{*}}^{WOT},\]
which is diffuse. Hence, by Property \ref{I:subadditivity of 1 bdd ent}, we have that $h(\widehat{Q}\vee P:M)\leq 0$, and by maximality, we have that $\widehat{Q}\vee P\subseteq P$. It follows that $\widehat{Q}\subseteq P$.
 So
\[Q=p\widehat{Q}p\subseteq pPp.\]

(\ref{item: amplification of Pinkser}):
View $M\leq M_{n}(M)$ by identifying it with $M\otimes 1\leq M\otimes M_{n}(\bC)\cong M_{n}(M)$. Under this identification, $\cN_{M_{n}(M)}(P)\supseteq \cU(M_{n}(\bC))\cup \cU(P)$, so $W^{*}(\cN_{M_{n}(M)}(P))\supseteq M_{n}(P)$. Thus, by Properties P\ref{item:preservation under normalizers}, P\ref{I:monotonicity of 1 bounded entropy}:
\[h(M_{n}(P):M_{n}(M))\leq h(P:M_{n}(M))\leq h(P:M)\leq 0. \]
So we can let $Q$ be the Pinsker algebra of $M_{n}(M)$ containing $M_{n}(P)$.  Let $e_{ij}$ be the standard matrix units of $M_{n}(\bC)$ viewed as elements of $M_{n}(M)$. Then, by (\ref{item: compressions of Pinsker}), we have that $e_{11}Qe_{11}=P$. But then, for all $x\in Q$ we have that
\[x=\sum_{i,j}e_{ii}xe_{jj}=\sum_{i,j}e_{i1}e_{11}xe_{11}e_{1j}\in M_{n}(P).\]
So $Q\leq M_{n}(P)$.

\end{proof}

\begin{proof}[Proof of Corollary \ref{Pinsker algebras of interpolated fgf}]
The forward implication is P4 of the main properties of $1$-bounded entropy we listed above. For the reverse implication, suppose for the sake of contradiction that $Q$ is not amenable. Then by Connes-Haagerup characterization of amenability (see Lemma 2.2 in \cite{HaagerupAmenable}), there is a projection $p\in Z(Q)$ and $u_{1},\cdots,u_{r}\in \mathcal{U}(Qp)$ so that
\[\left\|\frac{1}{r}\sum_{j=1}^{r}u_{j}\otimes \overline{u_{j}}\right\|<1.\]
Where $\overline{u_{j}}=(u_{j}^{*})^{op}$ and the norm is computed in $Qp\otimes_{\textnormal{min}}Qp$.
Let $P\leq pL(\bF_{t})p$ be the Pinsker algebra of $L(\bF_{t})$ which contains $Q$. By fundamental results of Dykema and R\u{a}dulescu (see \cite{DykemaIFG}, \cite{RadulaescuIFG}), we may choose $s>0$ so that $(pL(\bF_{t})p)^{s}\cong L(\bF_{2})$. Fix an integer $n>s$, and let $q$ be a projection in $M_{n}(P)$ so that $\Tr\otimes \tau(q)=s$. Observe that
\[\left\|\frac{1}{r}\sum_{j=1}^{r}(1_{n}\otimes u_{j})\otimes \overline{1_{n}\otimes u_{j}}\right\|<1,\]
where $1_{n}$ is the identity of $M_{n}(\bC)$. Hence, it follows that $M_{n}(P)$ also has no nonzero amenable direct summands. We leave it as an exercise to show that this implies that $M_{n}(P)$ has no nonzero amenable corners.
By Proposition \ref{prop: amplifications of Pinsker}, we know that
\[qM_{n}(P)q\leq qM_{n}(pL(\bF_{t})p)q\cong L(\bF_{2})\]
is Pinsker. It follows from Theorem \ref{thm:classification of Pinsker algebras} that $qM_{n}(P)q$ is amenable. This contradicts our previous observation that $M_{n}(P)$ has no nonzero amenable corners.
\end{proof}

\section{Main results}

\subsection{Orthocomplement bimodule structure for maximal amenable subalgebras}
\label{sec:orthocomp}

We start with the following consequence of Theorem \ref{thm:classification of Pinsker algebras} on the structure of the orthocomplement bimodule for any maximal amenable $P\leq L(\bF_{t})$. Note that this verifies the coarseness conjecture, independently due to the first-named author \cite[Conjecture 1.12]{Hayes8} and Popa \cite[Conjecture 5.2]{PopaWeakInter}.

\begin{cor}\label{cor: coarseness conjuecture}
Let $M=L(\bF_{t})$ for some $t>1$. For any maximal amenable $P\leq L(\bF_{t})$ we have that
\[_{P}(L^{2}(M)\ominus L^{2}(P))_{P}\leq [L^{2}(P)\otimes L^{2}(P)]^{\oplus \infty}.\]
\end{cor}

\begin{proof}
As noted in Proposition \ref{prop:singular subspaces contain weak intertwiners}, we have that
$\cH_{\textnormal{anti-c}}(P\leq M)^{\perp}$ embeds into $[L^{2}(P)\otimes L^{2}(P)]^{\oplus \infty}$
as a $P$-$P$ bimodule. Since $P$ is Pinsker by Theorem \ref{thm:classification of Pinsker algebras}, P\ref{item:preservation under normalizers} implies that
\[\cH_{\textnormal{anti-c}}(P\leq M)=L^{2}(P).\]
Thus,
\[
L^{2}(M) \ominus L^{2}(P)=\cH_{\textnormal{anti-c}}(P\leq M)^{\perp}
\]
embeds into $[L^{2}(P)\otimes L^{2}(P)]^{\oplus \infty}$.
\end{proof}



 Suppose $(M,\tau)$ is a tracial von Neumann algebra, and $A\leq M$ is a maximal abelian $*$-subalgebra. Write $A=L^{\infty}(X,\mu)$ for some compact Hausdorff space $X$ and some Borel probability measure $\mu$ on $X$. Let $\pi\colon C(X)\otimes C(X)\to B(L^{2}(M)\ominus L^{2}(A))$ be as in the definition of the left/right measure given in the introduction.
Note that if $\nu$ is a left/right measure, and if $\phi\colon C(X)\otimes C(X)\to L^{\infty}(X\times X,\nu)$ is the map sending an element of $C(X)\otimes C(X)\cong C(X\times X)$ to its $L^{\infty}(\nu)$-equivalence class, then there is a unique normal $*$-isomorphism $\rho\colon L^{\infty}(X\times X,\nu)\to \overline{\pi(C(X)\otimes C(X))}^{SOT}$ so that
$\pi=\rho\circ \phi.$

\begin{cor}\label{left/right measure is ac wrt my tensor mu}
Let $M=L(\bF_{t})$ for some $t>1$. Suppose that $A\leq M$ is abelian and a maximal amenable subalgebra of $M$. Write $A=L^{\infty}(X,\mu)$ for some compact metrizable space $X$ and some Borel probability measure $\mu$ on $X$. Then the left/right measure of $A\leq M$ is absolutely continuous with respect to $\mu\otimes \mu$.
\end{cor}

\begin{proof}
Let $E\colon X\times X\to\Proj(L^{2}(M)\ominus L^{2}(A))$ be the spectral measure corresponding to the representation $\pi$ defined as in the paragraph before Corollary \ref{left/right measure is ac wrt my tensor mu}. For a bounded, Borel map $\phi\colon X\times X\to \bC$, we let
\[\widetilde{\pi}(\phi)=\int \phi\,dE.\]
By Corollary \ref{cor: coarseness conjuecture}, we know that $L^{2}(M)\ominus L^{2}(A)$ embeds into an infinite direct sum of $L^{2}(A)\otimes L^{2}(A)$ as an $A$-$A$ bimodule.  Thus for any vector $\xi\in L^{2}(M)\ominus L^{2}(A)$, we may find a sequence $(k_{n})_{n}\in L^{2}(X\times X)$ so that $\sum_{n}\int\|k_{n}\|_{2}^{2}<+\infty$ and
\begin{align*}
\ip{\pi(\phi)\xi,\xi}&=\sum_{n}\ip{\phi k_{n},k_{n}}\\
&=\sum_{n}\int \int \phi(x,y)|k_{n}(x,y)|^{2}\,d\mu(x)\,d\mu(y) \textnormal{ for all $\phi\in C(X\times X)$.}
\end{align*}
Set $K=\sum_{n}|k_{n}|^{2}$.
Then for every bounded, Borel $\phi\colon X\times X\to \bC$ we have
\[\ip{\widetilde{\pi}(\phi)\xi,\xi}=\int \int \phi(x,y)K(x,y)\,d\mu(x)\,d\mu(y).\]
In particular, if $B\subseteq \bC$ is Borel, and $(\mu\otimes \mu)(B)=0$, then
\[\|\widetilde{\pi}(1_{B})\xi\|_{2}^{2}=\ip{\widetilde{\pi}(1_{B})\xi,\xi}=0,\]
the first equality holding as $\widetilde{\pi}(1_{B})$ is a projection. Since this holds for every $\xi\in L^{2}(M)\ominus L^{2}(A)$ we see that $E(B)=\widetilde{\pi}(1_{B})=0$. So we have shown that $E$ is absolutely continuous with respect to $\mu\otimes \mu$.
\end{proof}

\subsection{Generalizations of strong solidity}\label{sec: omega-mega solidity}

Throughout this section, we will be interested in properties of $W^{*}(\cH_{\textnormal{anti-c}}(Q\leq M))$ for an inclusion $Q\leq M$ of tracial von Neumann algebras. Since $\cH_{\textnormal{anti-c}}(Q\leq M)$ is a subset of $L^{2}(M)$ and not of $M$, we need to explain what we mean by $W^{*}(\cH_{\textnormal{anti-c}}(Q\leq M))$. Every $\xi\in L^{2}(M)$ may be identified with the densely defined, closed, unbounded operator $L_{\xi}$ on $L^{2}(M)$; this $L_\xi$ is the closure of the operator $L_{\xi}^{o}$ which has $\dom(L_{\xi}^{o})=M$ and is defined by $L_{\xi}^{o}(x)=\xi x$ for all $x\in M$. For $\xi\in L^{2}(M)$, we let $L_{\xi}=V_{\xi}|L_{\xi}|$ be its polar decomposition. For $X\subseteq L^{2}(M)$, we then define
\[W^{*}(X)=W^{*}\bigl(\{V_{\xi}:\xi\in X\}\cup \{\phi(|L_{\xi}|):\xi\in X,\phi\colon [0,\infty)\to \bC \textnormal{ is bounded and Borel}\}\bigr).\]

Throughout this section, given a tracial von Neumann algebra $(M,\tau)$ we view $M\leq M^{\omega}$ by identifying it with the image of the constant sequences.

\begin{defn}
Let $(M,\tau)$ be a tracial von Neumann algebra. For a free ultrafilter $\omega\in\beta \bN\setminus \bN$, we say that $M$ is \emph{$\omega$-strongly solid} if $W^{*}(N_{M^{\omega}}(Q))\cap M$ is amenable for all diffuse, amenable $Q\leq M^{\omega}$. We say that $M$ is \emph{ultrasolid} if it is $\omega$-strongly solid for every free ultrafilter $\omega$.
We say that $M$ is \emph{spectrally solid} if for any diffuse, amenable $Q\leq M$ we have that $W^{*}(\cH_{\textnormal{anti-c}}(Q\leq M))$ is amenable. Given a free ultrafilter $\omega\in\beta\bN\setminus \bN$ we say that $M$ is \emph{spectrally $\omega$-solid} if for any diffuse, amenable $Q\leq M^{\omega}$ we have that $W^{*}(\cH_{\textnormal{anti-c}}(Q\leq M^{\omega}))\cap M$ is amenable. We say that $M$ is \emph{spectrally ultrasolid} if it is \emph{spectrally $\omega$-solid} for every free ultrafilter $\omega$.
\end{defn}

\begin{cor} \label{cor: spectrally ultrasolid}
Fix $t>1$.
\begin{enumerate}[(i)]
\item $L(\bF_{t})$ is spectrally ultrasolid.  \label{item:spectral ultrasolidity}
\item If $Q\leq L(\bF_{t})$,
$\omega\in\beta \bN\setminus\bN$ is a free ultrafilter and $Q'\cap L(\bF_{t})^{\omega}$ is diffuse, then $Q$ is amenable. \label{item:Gamma absorbtion}
\end{enumerate}
\end{cor}

\begin{proof}
For notational simplicity, set $M=L(\bF_{t})$.

(\ref{item:spectral ultrasolidity})
Fix $\omega\in \beta\bN\setminus \bN$.
Let $Q\leq M^{\omega}$ be diffuse and amenable. Then by Properties P\ref{I:omegafying in the second variable}, P\ref{I:monotonicity of 1 bounded entropy}, P\ref{item:preservation under normalizers}, and P\ref{I:hyperfinite has 1-bdd ent zero}:
\begin{align*}
 h(W^{*}(\cH_{\textnormal{anti-c}}(Q\leq M^{\omega}))\cap M:M)&=h(W^{*}(\cH_{\textnormal{anti-c}}(Q\leq M^{\omega}))\cap M:M^{\omega})\\
&\leq h(W^{*}(\cH_{\textnormal{anti-c}}(Q\leq M^{\omega})):M^{\omega})\\
&=h(Q:M^{\omega})\leq 0.
\end{align*}

So $h(W^{*}(\cH_{\textnormal{anti-c}}(Q\leq M^{\omega}))\cap M:M)\leq 0$ which implies, Theorem \ref{cor:classification of Pinsker},  that $W^{*}(\cH_{\textnormal{anti-c}}(Q\leq M^{\omega}))\cap M$ is amenable.

(\ref{item:Gamma absorbtion}):
Fix $A\leq Q'\cap M^{\omega}$ diffuse and abelian. Then $Q\leq W^{*}(\cN_{M^{\omega}}(A))\leq W^{*}(\cH_{\textnormal{anti-c}}(A\leq M^{\omega}))$, so this follows from (\ref{item:spectral ultrasolidity}).
\end{proof}

In particular, Corollary \ref{cor: spectrally ultrasolid} and Proposition \ref{prop:singular subspaces contain weak intertwiners} imply that if $P\leq L(\bF_{t})$ is a maximal amenable subalgebra, then $wI_{L(\bF_{t})}(P,P)\subseteq L^{2}(P)$, and so $P$ is \emph{strongly malnormal} in the sense of Popa \cite{PopaWeakInter}.

We can take several iterations of this procedure in the ultraproduct setting and it will still have amenable intersection with the diagonal copy of $L(\bF_{t})$.

\begin{cor}\label{cor: fun with ordinals}
Fix $t>1$, and a free ultrafilter $\omega\in \beta\bN\setminus\bN$. Suppose that $Q\leq L(\bF_{t})^{\omega}$ is diffuse and amenable. Suppose we are given von Neumann subalgebras $Q_{\alpha}$ defined for ordinals $\alpha$ which satisfy the following properties:
\begin{itemize}
    \item $Q_{0}=Q$,
    \item if $\alpha$ is a successor ordinal, then $Q_{\alpha-1}\leq Q_{\alpha}\leq W^{*}(\cH_{\textnormal{anti-c}}(Q_{\alpha-1}\leq L(\bF_{t})^{\omega}))$,
    \item if $\alpha$ is a limit ordinal, then $Q_{\alpha}=\overline{\bigcup_{\beta<\alpha}Q_{\beta}}^{SOT}$.
\end{itemize}
Then, for any ordinal $\alpha$ we have that $Q_{\alpha}\cap L(\bF_{t})$ is amenable.
\end{cor}

\begin{proof}
One applies Properties P\ref{I:omegafying in the second variable}, P\ref{I:monotonicity of 1 bounded entropy}, P\ref{item:preservation under normalizers}, P\ref{I:increasing limits of 1bdd ent first variable}, and transfinite induction to see that
\[h(Q_{\alpha}:L(\bF_{t})^{\omega})=0\]
for any $\alpha$. It then follows by Property P\ref{I:omegafying in the second variable} that
\[h(Q_{\alpha}\cap L(\bF_{t}):L(\bF_{t})=h(Q_{\alpha}\cap L(\bF_{t}):L(\bF_{t})^{\omega})=h(Q_{\alpha}\cap L(\bF_{t}):L(\bF_{t})^{\omega})\leq h(Q_{\alpha}:L(\bF_{t})^{\omega})=0.\]
The Corollary now follows from Theorem \ref{thm:classification of Pinsker algebras}.
\end{proof}

\begin{cor}
Fix $t>1$. Then $L(\bF_{t})$ has the following Gamma stability property. Fix a free ultrafilter $\omega\in \beta\bN\setminus\bN$. If $Q\leq L(\bF_{t})^{\omega}$
and $Q'\cap L(\bF_{t})^{\omega}$ is diffuse, then $Q\cap L(\bF_{t})$ is amenable.
\end{cor}

\begin{proof}
Fix $A\leq Q'\cap L(\bF_{t})^{\omega}$ diffuse and abelian.
Note that
\[Q\leq W^{*}(\cN_{L(\bF_{t})^{\omega}}(A))\leq W^{*}(\cH_{\textnormal{anti-c}}(A\leq L(\bF_{t})^{\omega})).\]
Applying Corollary \ref{cor: spectrally ultrasolid} with $Q_{0}=A$, and $Q_{\alpha}=Q\vee A$ for all $\alpha\geq 1$ we see that $Q\cap L(\bF_{t})\leq (A\vee Q)\cap L(\bF_{t})$ is amenable.
\end{proof}

\subsection{Intertwining properties for Pinsker algebras}\label{sec:intertwine}

In this section, we explore how Pinsker algebras behave with respect to intertwining properties in the sense of Popa \cite{PopaL2Betti,PopaStrongRigidity}.

\begin{thm} \label{thm: Pinsker dichotomy}
Let $(M,\tau)$ be a tracial von Neumann algebra, and let $P,Q\leq M$ be Pinsker. Then exactly one of the following occurs:
\begin{itemize}
    \item either there are nonzero projections $e\in P,f\in Q$ and a unitary $u\in \mathcal{U}(M)$ so that $u(ePe)u^{*}=fQf$,
    \item or $wI_{M}(Q,P)=\{0\}$.
\end{itemize}
\end{thm}

\begin{proof}
Suppose that $wI_{M}(Q,P)\neq \{0\}$, so there is a diffuse $Q_{0}\leq Q$ with $Q_{0}\prec P$. This means there are nonzero projections $f_{0}\in Q_{0},e_{0}\in P$, a unital $*$-homomorphism $\Theta\colon f_{0}Q_{0}f_{0}\to e_{0}Pe_{0}$, and a nonzero partial isometry $v\in M$ so that:
\begin{itemize}
    \item $xv=v\Theta(x)$ \textnormal{ for all $x\in f_{0}Q_{0}f_{0}$},
    \item $vv^{*}\in (f_{0}Q_{0}f_{0})'\cap f_{0}Mf_{0},$
    \item $v^{*}v\in \Theta(f_{0}Q_{0}f_{0})'\cap e_{0}Me_{0}$.
\end{itemize}
By Property P\ref{item:preservation under normalizers}, we have  $h(W^{*}(\cN_{f_{0}Mf_{0}}(f_{0}Q_{0}f_{0})):f_{0}Mf_{0})\leq h(f_{0}Qf_{0}:f_{0}Mf_{0})\leq 0$ and since $f_{0}Qf_{0}\leq f_{0}Mf_{0}$ is Pinsker by Proposition \ref{prop: amplifications of Pinsker}, we know that $\cN_{f_{0}Mf_{0}}(f_{0}Q_{0}f_{0})\subseteq \cU(f_{0}Qf_{0})$. Similarly, $\cN_{e_{0}Me_{0}}(e_{0}Pe_{0})=\cU(e_{0}Pe_{0})$. It then follows as in \cite[Theorem 6.8]{Pineapple} that $f=vv^{*}\in Q, e=v^{*}v\in P$.

Since $f\in Q$, we have that $v^{*}(fQf)v$ is a subalgebra of $eMe$. Moreover, conjugation by $v$ implements an isomorphism between the inclusion $fQf\leq fMf$ and the inclusion $v^{*}(fQf)v\leq eMe,$ which implies that $h(v^{*}(fQf)v:eMe)\leq 0$. Moreover,
\[v^{*}(fQf)v\cap ePe\supseteq e\Theta(f_{0}Qf_{0}).\]
Since $\Theta(f_{0}Qf_{0})$ is the image of a diffuse subalgebra under a normal $*$-homomorphism, it follows that it is diffuse.
Since $ePe$ is Pinsker by Proposition \ref{prop: amplifications of Pinsker} (\ref{item: compressions of Pinsker}), and $h(v^{*}(fQf)v:eMe)\leq 0$, this forces $v^{*}(fQf)v=ePe$ by Property P\ref{I:subadditivity of 1 bdd ent}. Since $M$ is finite, there is a unitary $u\in \mathcal{U}(M)$ so that $fu=v$. Then $u^{*}(fQf)u=ePe$, as desired.
\end{proof}

In the case where $P$ and $Q$ are factors, the first option in the dichotomy can be strengthened to saying that $P$ and $Q$ are unitarily conjugate.

\begin{cor} \label{cor: Pinsker dichotomy factor case}
Let $(M,\tau)$ be a tracial von Neumann algebra.  Let $P$ and $Q$ be Pinsker subalgebra of $M$ such that $P$ and $Q$ are factors.  Then either $P$ and $Q$ are unitarily conjugate or $wI_M(P,Q) = 0$.
\end{cor}

This follows from the general fact that if $M$ is a tracial von Neumann algebra and $P$ and $Q \leq M$ are factors with unitarily conjugate corners, then they are unitarily conjugate.  This is a folklore result and we include the proof here for completeness.

\begin{prop}
If $Q,P$ are subalgebras of a tracial von Neumann algebra $(M,\tau)$ with unitarily conjugate corners, and if $P,Q$ are factors, then $P,Q$ are unitarily conjugate.
\end{prop}

\begin{proof}
Choose nonzero projections $p\in P,q\in Q$ and a unitary partial isometry $v\in M$ with $v^{*}v=p,vv^{*}=q$ and
\[vPv^{*}=qQq.\]
Since $P$ is a factor, we may shrink $p,q$ if necessary to assume that $\tau(p)=1/n$ for some integer $n$.
Choose projections $p_{1},\cdots,p_{n}$ in $P$ which are pairwise orthogonal with $\tau(p_{j})=1/n$ for all $j$, and with $p_{1}=p$. Choose analogous projections $q_{1},\cdots,q_{n}$ in $Q$ with $q_{1}=q$.  Since $P,Q$ are factors for $2\leq j\leq n$ we may choose partial isometries $a_{j},b_{j}$ in $P,Q$ so that $a_{j}^{*}a_{j}=p,a_{j}a_{j}^{*}=p_{j}$, $b_{j}^{*}b_{j}=q,b_{j}b_{j}^{*}=q_{j}$, and set $a_{1}=p,b_{1}=q$.
Finally, let
\[u=\sum_{i}b_{i}va_{i}^{*}.\]
Then $u$ is a unitary, and if $x\in P$, then $u(a_{i}a_{i}^{*}xa_{j}a_{j}^{*})u^{*}\in Q$ for all $1\leq i,j\leq n$. Using that any $x\in P$ is equal to $\sum_{i,j}a_{i}a_{i}^{*}xa_{j}a_{j}^{*}$ we see that $uPu^{*}\subseteq Q$. By a symmetric argument, $u^{*}Qu\subseteq P$.
\end{proof}



The combination of the above results allows us to deduce Theorem \ref{thm:Pinkser dichotomy intro} as follows.

\begin{proof}[Proof of Theorem \ref{thm:Pinkser dichotomy intro}]
The fact that either (\ref{item:option 1 intro}) or (\ref{item:option 2 intro}) of Theorem \ref{thm:Pinkser dichotomy intro} hold follows from Theorems \ref{thm: Pinsker dichotomy} and \ref{thm:classification of Pinsker algebras}.
The ``in particular" part follows from Corollary \ref{cor: Pinsker dichotomy factor case} and Theorem \ref{thm:classification of Pinsker algebras}.
\end{proof}

\begin{example} \label{ex: no intertwining 1}
The work of Dykema \cite{DykemaFreeProductsHyper} implies that $L(\bF_2) \cong L^\infty[0,1] * \mathcal{R}$.  By the work of Popa \cite{PopaMaximalAmenable}, $L^\infty[0,1]$ and $\mathcal{R}$ are maximal amenable subalgebras in $L(\bF_2)$.  Hence, they are Pinsker subalgebras by Theorem \ref{thm:classification of Pinsker algebras}.  (Alternatively, \cite{FreePinsker} shows directly that they are Pinsker subalgebras.)  Therefore, given any automorphism $\phi$ of $L(\bF_2)$, by Theorem \ref{thm: Pinsker dichotomy}, $L^\infty[0,1]$ and $\phi(\mathcal{R})$ in $L(\bF_2)$ either have zero intertwiners or have unitarily conjugate corners.  They do not have isomorphic corners, and therefore  $wI(L^\infty[0,1],\phi(\mathcal{R})) = 0$.
\end{example}

Since the free product of any two amenable separable tracial von Neumann algebras is $L(\bF_{2})$, we can generalize this example quite a bit.  To handle these more general examples, we want a strengthening of Theorem \ref{thm: Pinsker dichotomy} along the lines of Corollary \ref{cor: Pinsker dichotomy factor case} that does not assume that $P$ and $Q$ are factors.  First, we use a maximality argument to make the projection $e$ in Theorem \ref{thm: Pinsker dichotomy} as large as possible.

\begin{thm} \label{thm: Pinsker dichotomy 2}
Let $(M,\tau)$ be a tracial von Neumann algebra, and let $P, Q \leq M$ be Pinsker subalgebras.  There exist projections $e \in P$ and $f \in Q$, and a partial isometry $v \in M$ from $e$ to $f$, such that the following hold:
\begin{enumerate}
    \item $v(ePe)v^* = fQf$.
    \item If $e \neq 1$ and $u$ is a unitary such that $v = ue = fu$, then
    \[
    wI_{(1-f)M(1-f)}((1-f)Q(1-f),u(1 - e)P(1 - e)u^*) = 0.
    \]
\end{enumerate}
\end{thm}

\begin{proof}
\textbf{Step 1:} We show that there exists a choice of $e$ and $f$ satisfying (1) that is maximal, in the sense that no strictly larger projections satisfy (1).  To this end, we will apply Zorn's lemma to the set of triples $(e,f,v)$ where $e \in P$ and $f \in Q$ are projections and $v$ is a partial isometry from $e$ to $f$ such that $v(ePe)v^* = fQf$.  The partial order on this set will be given by $(e,f,v) \leq (e',f',v')$ if $e \leq e'$, $f \leq f'$,
and $fv'e = v$.
Note that $fv'e=v$ implies that $e=v^{*}v=e(v')^{*}fv'e$ which in turn implies
\[|v'e-v|^{2}=e+e(v')^{*}v'e-2\rea(e(v')^{*}v)=e+e(v')^{*}v'e-2\rea(e(v')^{*}fv'e)=e(v')^{*}v'e-e\leq 0.\]
So $v'e=v$ and similarly $fv'=v$.
One checks readily that the above order is a partial order.  So it remains to show that every increasing chain $\{(e_\alpha,f_\alpha,v_\alpha)\}_{\alpha \in I}$ has an upper bound.  Let $e = \sup_\alpha e_\alpha$ and $f = \sup_\alpha f_\alpha$.  Note that for $\alpha \leq \beta$, we have
\[v_{\beta}-v_{\alpha}=v_{\beta}e_{\beta}-v_{\alpha}=v_{\beta}(e_{\beta}-e_{\alpha}),\]
using our previous observation that $\alpha\leq \beta$ implies that $v_{\alpha}=v_{\beta}e_{\alpha}$.
Since $e_\alpha$ converges to $e$,
 it follows that $(e_\alpha)_{\alpha \in I}$
 is Cauchy in $L^2(M)$; this in turn implies that $(v_\alpha)_{\alpha \in I}$ is Cauchy in $L^2(M)$ and hence converges to some $v \in L^2(M)$.  This $v$ is necessarily also a partial isometry, and $v_\alpha = f_\alpha v e_\alpha$.
Moreover, we have $v (ePe) v^* \subseteq fQf$ because for each $x \in P$, we have $v(e_\alpha x e_\alpha) v^* = v_\alpha (e_\alpha x e_\alpha) v_\alpha^* \in f_\alpha Q f_\alpha \subseteq f Q f$.  So taking the limit over $\alpha$, we get $v(exe)v^* \in fQf$.  The same reasoning shows that $v^* fQf v \subseteq ePe$, and hence $v(ePe) v^* = fQf$.  Hence, by Zorn's lemma, there exists a maximal choice of $e$, $f$, and $v$.

\textbf{Step 2:}  Now we will apply Theorem \ref{thm: Pinsker dichotomy} to show that the maximal $e$, $f$, and $v$ satisfy (2).  Let $u$ be a unitary such that $v = ue = fu$.  Suppose for contradiction that $wI_{(1-f)M(1-f)}((1-f)Q(1-f),u(1-e)P(1-e)u^*) \neq \{0\}$.  By Proposition \ref{prop: amplifications of Pinsker} (ii), $(1-f)Q(1-f)$ is Pinsker in $(1-f)M(1-f)$.  Note $uPu^*$ is Pinsker in $M$ and hence $(1-f)uPu^*(1-f) = u(1-e)P(1-e)u^*$ is Pinsker in $(1-f)M(1-f)$.  By Theorem \ref{thm: Pinsker dichotomy}, $wI_{(1-f)M(1-f)}((1-f)Q(1-f),u(1-e)P(1-e)u^*) \neq \{0\}$ implies that there exist projections $e_0 \in u(1-e)P(1-e)u^*$ and $f_0 \in (1-f)Q(1-f)$ and a partial isometry $v_0$ from $e_0$ to $f_0$ that conjugates $e_0u(1-e)P(1-e)u^*e_0$ to $f_0(1 - f)Q(1-f)f_0$.  Write $e' = u^*e_0u$, so that $e'$ is a projection in $P$ with $e' \leq 1-e$.  Let $f' = f_0$ which is a projection in $Q$ with $f' \leq 1 -f $.  Let $v' = v_0 u$, which is a partial isometry in $M$ sending $e'$ to $f'$.  Then
\[
v'e'Pe'(v')^* =v_{0} e_0u(1 - e)P(1 - e)u^*e_0 v_{0}^{*} = f_{0}(1-f)Q(1-f)f_{0} = f'Qf'.
\]
By Proposition \ref{prop: amplifications of Pinsker} (ii), $(f + f')Q(f + f')$ is Pinsker in $(f+f')M(f+f')$, and $(e+e')P(e+e')$ is Pinsker in $(e+e')M(e+e')$, so that $(v+v')(e+e')P(e+e')(v+v')^* = (f+f')(v+v')P(v+v')^*(f+f')$ is Pinsker in $(f+f')M(f+f')$.  Now $(f+f')Q(f+f')$ and $(v+v')(e+e')P(e+e')(v+v')^*$ contain the common diffuse subalgebra
\[
fQf \oplus f'Qf' = (v+v')[ePe \oplus e'Pe'](v+v')^*.
\]
Since $(f+f')Q(f+f')$ and $(v+v')(e+e')P(e+e')(v+v')^*$ are both Pinsker in $(f+f')Q(f+f')$ and intersect diffusely, they must be equal.  This contradicts the maximality of $(e,f,v)$, and thus establishes that $wI_{(1-f)M(1-f)}((1-f)Q(1-f),u(1-e)P(1-e)u^*) = \{0\}$.
\end{proof}

Theorem \ref{thm: Pinsker dichotomy 2} implies the following corollary about projections in the Pinsker algebras $P$ and $Q$.  Note that in the case where $P$ and $Q$ are both factors, (2) below reduces to saying $e$ is either $0$ or $1$, which yields Corollary \ref{cor: Pinsker dichotomy factor case}.  Thus, the following corollary can be understood as a generalization of Corollary \ref{cor: Pinsker dichotomy factor case}.

\begin{cor}
Let $P$ and $Q$ be Pinsker subalgebras of a tracial von Neumann algebra $(M,\tau)$. Let $e$, $f$, and $v$ be as in the previous theorem.  Then
\begin{enumerate}
    \item Let $e_1, e_2 \in P$ with $e_1 \leq 1 - e$, $e_2 \leq e$ and $e_1 \sim_P e_2$.  Let $f_1$ and $f_2$ satisfy the analogous conditions for $Q$ and $f$.  Then $ve_2v^* \wedge f_2 = 0$.
    \item In particular if $Q$ is a factor, then $e$ is central in $P$.
\end{enumerate}
\end{cor}

\begin{proof}
(1) Suppose $e_1, e_2 \in P$ with $e_1 \leq 1 - e$, $e_2 \leq e$ and $e_1 \sim_P e_2$.  Let $f_1$ and $f_2$ satisfy the analogous conditions for $Q$ and $f$.  Suppose for contradiction that $ve_2v^* \wedge f_2 \neq 0$.  Let $f_2' = ve_2v^* \wedge f_2$.  Let $f_1'$ be the corresponding subprojection of $f_1$.  Let $e_2'
= v^*f_2'v$, and let $e_1'$ be the corresponding subprojection of $e_1$.  Then $e_1'Pe_1'$ is unitarily conjugate to $e_2'Pe_2'$ using the partial isometry that transforms $e_1$ to $e_2$, and similarly $f_1'Qf_1'$ is unitarily conjugate to $f_2'Qf_2'$.  And $e_2'Pe_2'$ is unitarily conjugate to $f_2'Qf_2'$ using $v$.  This implies that $wI_{(1-f)M(1-f)}((1-f)Q(1-f),v(1-e)P(1-e)v^*) \neq \{0\}$ (or alternatively, it directly contradicts the maximality in Step 2 of the previous proof).

(2) Suppose $Q$ is a factor, and assume for contradiction that $e$ is not central in $P$.  Then there must exist some projections $e_1, e_2 \in P$ with $e_1 \leq 1 - e$, $e_2 \leq e$ and $e_1 \sim_P e_2$.  Because $Q$ is a factor, there exist projections $f_1$ and $f_2$ satisfying the analogous conditions for $Q$ and $f$ and with $f_{2}=ve_{2}v^{*}$.  Hence, we get a contradiction from (1).
\end{proof}

\begin{example} \label{ex: Pinsker intertwining}
Dykema \cite[Theorem 4.6]{DykemaFreeProductsHyper} showed that if $A$ and $B$ are SOT-separable diffuse amenable tracial von Neumann algebras, then $A * B \cong L(\bF_{2})$.  Taking two such pairs, there is an isomorphism $\alpha: A_1 * B_1 \to A_2 * B_2 = M$; let $\alpha$ be any such isomorphism.  Note that $A_1$ and $A_2$ are Pinsker subalgebras by \cite{FreePinsker}, and hence Theorem \ref{thm: Pinsker dichotomy} applies to $A_1$ and $A_2$.

\begin{itemize}
    \item In particular, suppose that $A_1 = \cR$ and $A_2 = \cR$.  Then either $\alpha(A_1)$ and $A_2$ are unitarily conjugate, or $wI_M(A_2,\alpha(A_1)) = \{0\}$.
    \item Suppose that $A_1 = \cR \oplus \cR$ and $A_2 =\cR$.  Then the projection $e$ from Theorem \ref{thm: Pinsker dichotomy 2} must be central in $A_1$, and hence there are only four possible choices of $e$, resulting in a tetrachotomy.
    \item Generalizing Example \ref{ex: no intertwining 1}, suppose $A_1 = L^\infty[0,1]$ and $A_2 = \cR$.  Then for any nonzero projections $e \in \alpha(A_1)$ and $f \in A_2$, the von Neumann algebras $e\alpha(A_1)e$ and $fA_2f$ are not isomorphic.  Hence the projection $e$ in Theorem \ref{thm: Pinsker dichotomy 2} must be zero, and hence $wI_M(A_2,\alpha(A_1)) = \{0\}$.
\end{itemize}
\end{example}

\appendix

\section{Invariance of $1$-bounded entropy via noncommutative functional calculus}

\subsection{Microstate spaces and definition of $h$}

Here we recall definitions of the space of non-commutative laws.  Let $\bC\ip{t_i:i\in I}$ be the algebra of non-commutative complex polynomials in $(t_i)_{i \in I}$ (i.e. the free $\bC$-algebra on the set $I$). We give $\bC\ip{t_i:i\in I}$ the unique $*$-algebra structure which makes the $t_i$ self-adjoint.  If $\cM$ is a $\mathrm{W}^*$-algebra and $\mathbf{x} = (x_i)_{i \in I}\in \cM_{sa}^I$
, then there is a unique $*$-homomorphism $\ev_\mathbf{x}: \bC\ip{t_i: i \in I} \to \cM$ such that $\ev_{\mathbf{x}}(t_i) = x_i$.  For a non-commutative polynomial $p \in \bC\ip{t_i: i \in I}$, we define $p(\mathbf{x}) = p((x_i)_{i \in I})$ to be $\ev_\mathbf{x}(p)$.

A tracial \emph{non-commutative law} of a self-adjoint $I$-tuple is a linear functional $\lambda: \bC\ip{t_i: i \in I} \to \bC$ that is
\begin{enumerate}
	\item unital, that is, $\lambda(1) = 1$;
	\item positive, that is, $\lambda(p^*p) \geq 0$;
	\item tracial, that is, $\lambda(pq) = \lambda(qp)$;
	\item exponentially bounded, that is, for some $(R_i)_{i \in I} \in (0,+\infty)^I$, we have
	\[
	|\lambda(t_{i(1)} \dots t_{i(\ell)})| \leq R_{i(1)} \dots R_{i(\ell)}
	\]
	for all $\ell$ and all $i(1)$, \dots, $i(\ell) \in I$.
\end{enumerate}
Given $\mathbf{R} = (R_i)_{i \in I} \in (0,+\infty)^I$, we define $\Sigma_{\mathbf{R}} = \Sigma_{(R_i)_{i \in I}}$ to be the set of non-commutative laws satisfying (4) for our given choice of $(R_i)_{i \in I}$.  We equip $\Sigma_{\mathbf{R}}$ with the topology of pointwise convergence on $\bC\ip{t_i: i \in I}$, which makes it a compact Hausdorff space.

For a tracial $\mathrm{W}^*$-algebra $(\cM, \tau)$, a tuple $\mathbf{x} = (x_i)_{i \in I} \in \cM_{sa}^I$ and $\mathbf{R} = (R_i)_{i \in I} \in (0,+\infty)^I$ satisfying $\norm{x_i} \leq R_i$, we define the \emph{non-commutative law of $\mathbf{x}$} as the map
\[
\lambda_{\mathbf{x}}: \bC\ip{t_i: i \in I} \to \bC: \quad p \mapsto \tau(p(x)).
\]
It is straightforward to verify that $\lambda_{\mathbf{x}}$ is in $\Sigma_{\mathbf{R}}$.  Conversely, given any $\lambda \in \Sigma_{\mathbf{R}}$, there exists some $(\cM, \tau)$ and $\mathbf{x} \in \cM_{sa}^I$ such that $\lambda_{\mathbf{x}} = \lambda$ and $\norm{x_i} \leq R_i$ for all $i \in I$; see either \cite[Proposition 4.2]{BNLaws} or the proof of \cite[Proposition 5.2.14(d)]{AGZ2009}.
We also remark that $(\cM, \tau)$ could be $M_n(\bC)$ with the normalized trace $\tau_n = (1/n) \Tr$.  Thus, if $\mathbf{x} \in M_n(\bC)_{sa}^I$, then $\lambda_{\mathbf{x}}$ is a well-defined non-commutative law.

The $1$-bounded entropy $h$ is defined in terms of Voiculescu's microstate spaces.

\begin{defn}[Microstate space]
Let $\cM$ be a tracial von Neumann algebra and $I$ an index set.  Let $\mathbf{R} \in (0,+\infty)^I$, let $\mathbf{y} \in \cM_{sa}^I$ be a self-adjoint tuple with $\norm{y_i} \leq R_i$, and let $\cO \subseteq \Sigma_{\mathbf{R}}$ be a neighborhood of $\lambda_{\mathbf{y}}$.  Then we define the microstate space
\[
\Gamma_{\mathbf{R}}^{(n)}(\cO) := \{\mathbf{Y} \in M_n(\bC)_{sa}^I: \norm{Y_i} \leq R_i \text{ for all } i \in I \text{ and } \lambda_{\mathbf{Y}} \in \cO \}.
\]
\end{defn}

\begin{defn}[Orbital covering numbers]
Let $I$ be an index set and let $\Omega \subset M_n(\bC)_{s.a.}^I$.  For $F \subseteq I$ finite and $\eps > 0$, we define the \emph{orbital $(F,\eps)$-neighborhood} of $\Omega$ as the set
\[
N_{F,\eps}^{\orb}(\Omega) = \{ \mathbf{Y} \in M_n(\bC)_{s.a.}^I: \textnormal{ there exists } \mathbf{Y}' \in \Omega, U \in \mathcal{U}(M_n(\bC)), \norm{Y_i - UY_i'U^*}_2 < \eps \text{ for } i \in F \}.
\]
Moreover, we define the orbital covering number $K_{F,\eps}^{\orb}(\Omega)$ as the minimal cardinality of a set $\Omega'$ such that $\Omega \subseteq N_{F,\eps}^{\orb}(\Omega')$.
\end{defn}

\begin{defn}
Let $(M,\tau)$ be a tracial von Neumann algebra, let $I$ and $J$ be index sets, let $\mathbf{R} \in (0,\infty)^I$ and $\mathbf{S} \in (0,\infty)^J$, and let $\mathbf{x} \in M_{s.a.}^I$ and $\mathbf{y} \in M_{s.a.}^J$ with $\norm{x_i} \leq R_i$ for $i \in I$ and $\norm{y_j} \leq S_j$ for $j \in J$.  For a neighborhood $\cO$ of $\lambda_{(\mathbf{x},\mathbf{y})}$ in $\Sigma_{(\mathbf{R},\mathbf{S})}$, consider the microstate space $\Gamma_{(\mathbf{R},\mathbf{S})}^{(n)}(\cO) \subseteq M_n(\bC)_{s.a.}^{I \sqcup J}$ and let $\pi_I(\Gamma_{(\mathbf{R},\mathbf{S})}^{(n)}(\cO))$ be its projection onto the $I$-indexed coordinates.  Then define
\[
h_{\mathbf{R},\mathbf{S}}(\mathbf{x}: \mathbf{y}) = \sup_{\substack{F \subseteq I \text{ finite} \\ \eps > 0}} \inf_{\cO \ni \lambda_{(\mathbf{x},\mathbf{y})}} \limsup_{n \to \infty} \frac{1}{n^2} \log K_{F,\eps}^{\orb}(\pi_I(\Gamma_{(\mathbf{R},\mathbf{S})}^{(n)}(\cO))).
\]
\end{defn}

In this appendix we will give an argument showing, at the same time, that this computation yields the same result for every $\mathbf{R}$ and $\mathbf{S}$ with $R_i \geq \norm{x_i}$ and $S_j \geq \norm{y_j}$, and that $h(\mathbf{x}:\mathbf{y})$ only depends on $\mathrm{W}^*(\mathbf{x})$ and $\mathbf{W}^*(\mathbf{x},\mathbf{y})$ (and the restriction of the trace to these algebras).

\subsection{$L^2$-continuous functional calculus}

Here we recall the construction of a certain space of non-commutative functions given as $L^2$-limits of trace polynomials, uniformly over all non-commutative laws.  Trace polynomials have been studied in many previous works such as \cite{Razmyslov1974,Procesi1976,Razmyslov1987,Rains1997,Sengupta2008,Cebron2013, DHK2013, Jing2015, Kemp2016, Kemp2017, DGS2016}.  The uniform $L^2$-completion of trace polynomials was first introduced in \cite{JekelConvexPot,JekelEAP,JekelThesis}, and its relationship with continuous model theory was addressed in \cite[\S 3.5]{Jekeltypes}.
 Moreover, \cite[Remark 3.5]{FreePinsker} described how this space is an example of the tracial completions of $\mathrm{C}^*$-algebras studied in \cite[p.\ 351-352]{Ozawa2013} and \cite{BBSTWW2019} and implicitly in \cite[\S 6]{CGSTW}.  Here we follow the version of the construction in \cite[\S 3]{FreePinsker}.

\begin{defn}
Fix an index set $I$ and $\mathbf{R} \in (0,+\infty)^I$.  Consider the space
\[
\mathcal{A}_{\mathbf{R}} = C(\Sigma_{\mathbf{R}}) \otimes \bC\ip{t_i: i \in I}.
\]
Given $(\cM,\tau)$ and $\mathbf{x} \in \cM_{sa}^I$ with $\norm{x_i} \leq R_i$, we define the evaluation map
    \begin{align*}
        \ev_{\mathbf{x}}: \mathcal{A}_{\mathbf{R}} &\to \cM\\
        \phi \otimes p &\mapsto \phi(\lambda_{\mathbf{x}}) p(\mathbf{x}).
    \end{align*}
Then we define a semi-norm on $C(\Sigma_{\mathbf{R}}) \otimes \bC\ip{t_i: i \in I}$ by
\[
\norm{f}_{\mathbf{R},2} = \sup_{(\cM,\tau), \mathbf{x}} \norm{\ev_{\mathbf{x}}(f)}_{L^2(\cM)},
\]
where the supremum is over all tracial $\mathrm{W}^*$-algebras $(\cM,\tau)$ and all $\mathbf{x}\in \cM_{sa}^I$ with $\norm{x_i} \leq R_i$. Denote by $\mathcal{F}_{\mathbf{R},2}$ the completion of $\mathcal{A}_{\mathbf{R}} / \{f \in \mathcal{A}_{\mathbf{R}}: \norm{f}_{\mathbf{R},2} = 0\}$.
\end{defn}

It is immediate that for every $(\cM,\tau)$, for every self-adjoint tuple $\mathbf{x} \in \cM_{sa}^I$ with $\norm{x_i} \leq R_i$, the evaluation map $\ev_{\mathbf{x}}: \mathcal{A}_{\mathbf{R}} \to \cM$ passes to a well-defined map $\mathcal{F}_{\mathbf{R},2} \to L^2(\cM)$, which we continue to denote by $\ev_{\mathbf{x}}$, and we will also define $f(\mathbf{x}) = \ev_{\mathbf{x}}(f)$.  Moreover, it is clear that $f(\mathbf{x}) = \ev_{\mathbf{x}}(f)$ always lies in $L^2(\mathrm{W}^*(\mathbf{x}))$ because this holds when $f$ is a simple tensor.

It will be convenient often to restrict our attention to elements of $\mathcal{F}_{\mathbf{R},2}$ that are bounded in operator norm.  For $f \in \mathcal{F}_{\mathbf{R},2}$, let us define
\[
\norm{f}_{\mathbf{R},\infty} = \sup_{(\cM,\tau), \mathbf{x}} \norm{\ev_{\mathbf{x}}(f)},
\]
and then set
\[
\mathcal{F}_{\mathbf{R},\infty} = \{f \in \mathcal{F}_{\mathbf{R},2}: \norm{f}_{\mathbf{R},\infty} < +\infty\}.
\]
\cite[Lemma 3.3]{FreePinsker} showed that $\mathcal{F}_{\mathbf{R},\infty}$ is a $\mathrm{C}^*$-algebra with respect the norm $\norm{\cdot}_{\mathbf{R},\infty}$ and the multiplication and $*$-operation arising from the natural ones on simple tensors.

One of the most useful properties of $\cF_{\mathbf{R},\infty}$ is that it allows all the elements of a von Neumann algebra to be expressed as a function of the generators.  More precisely, \cite[Proposition 3.4]{FreePinsker} showed the following:

\begin{prop}
Given $(\cM, \tau)$ and $\mathbf{x} \in \cM_{sa}^I$ with $\norm{x_i} \leq R_i$, the evaluation map $\ev_{\mathbf{x}}:\mathcal{F}_{\mathbf{R},2}
\to L^2(\mathrm{W}^*(\mathbf{x}))$ is surjective.  It also restricts to a surjective $*$-homomorphism $\mathcal{F}_{\mathbf{R},\infty} \to \mathrm{W}^*(\mathbf{x})$.
\end{prop}

This surjectivity property on its own is not too significant, because for instance, the double dual of the $\mathrm{C}^*$-universal free product of $C[-R_i,R_i]$ over $i \in I$ can be used to define a functional calculus that is surjective.  The benefit of the construction used here is that it achieves surjectivity at the \emph{same time} as relatively strong continuity properties.

First, we show that the non-commutative law of the output will depend continuously on the non-commutative law of the input. As we will see later, this property allows these functions to transform between microstate spaces for different generators of a von Neumann algebra.  To fix notation, let $I$ and $I'$ be index sets.  Let $\mathbf{R} \in (0,+\infty)^I$ and $\mathbf{R}' \in (0,+\infty)^{I'}$.  We define
\[
\mathcal{F}_{\mathbf{R},\mathbf{R}'} = \{ \mathbf{f} = (f_i)_{i \in I'} \in (\mathcal{F}_{\mathbf{R},\infty})_{sa}^{I'}: \norm{f_i}_{\mathbf{R},\infty} \leq R_i' \text{ for all } i\in I'\}.
\]

\begin{prop}[{\cite[Proposition 3.7]{FreePinsker}}] \label{prop:pushforwardcontinuity}
Let $\mathbf{R} \in (0,+\infty)^I$ and $\mathbf{R}' \in (0,+\infty)^{I'}$.  Let $\mathbf{f} = (f_i)_{i \in I'} \in \mathcal{F}_{\mathbf{R},\mathbf{R}'}$.
\begin{enumerate}
	\item Given $(\cM,\tau)$ and $\mathbf{x} \in \cM_{sa}^I$ with $\norm{x_i} \leq R_i$, we set $\mathbf{f}(\mathbf{x}) = (f_i(\mathbf{x}))_{i \in I'}$.  Then $\lambda_{\mathbf{f}(\mathbf{x})}$ is uniquely determined by $\lambda_{\mathbf{x}}$.
	\item Let $\mathbf{f}_*$ be the ``push-forward'' mapping $\Sigma_{\mathbf{R}} \to \Sigma_{\mathbf{R}'}$ defined by $\mathbf{f}_* \lambda_{\mathbf{x}} = \lambda_{\mathbf{f}(\mathbf{x})}$ for all such tuples $\mathbf{x}$.  Then $\mathbf{f}_*$ is continuous.
\end{enumerate}
\end{prop}

Another consequence of this is these spaces of functions are closed under composition.

\begin{prop} \label{prop: composition}
Fix index sets $I$, $I'$, and $I''$ and corresponding tuples $\mathbf{R}$, $\mathbf{R}'$, and $\mathbf{R}''$ from $(0,\infty)$.  Let $\mathbf{f} \in \mathcal{F}_{\mathbf{R},\mathbf{R}'}$ and $\mathbf{g} \in \mathcal{F}_{\mathbf{R}',\mathbf{R}''}$.  Then there exists a unique $\mathbf{h} \in \mathcal{F}_{\mathbf{R},\mathbf{R}'}$ such that $\mathbf{h}(\mathbf{x}) = \mathbf{g}(\mathbf{f}(\mathbf{x}))$ for all $M$ and $\mathbf{x} \in M^I$ with $\norm{x_i} \leq R_i$.
\end{prop}

\begin{proof}
First, we consider $g \in \mathcal{F}_{\mathbf{R}',2}$ and show that $g \circ \mathbf{f}$ is a well-defined element of $\mathcal{F}_{\mathbf{R}',2}$.  If $g$ is a simple tensor $\phi \otimes p$, then $\phi(\lambda_{\mathbf{f}(\mathbf{x})}) = \phi \circ \mathbf{f}_* \lambda_{\mathbf{x}}$ defines a continuous function on the space of laws by the previous proposition.  Also, since $\mathcal{F}_{\mathcal{R},\infty}$ is a $\mathrm{C}^*$-algebra, $p \circ \mathbf{f} \in \mathcal{F}_{\mathbf{R},\infty}$, and hence so is $g \circ \mathbf{f} = (\phi \circ \mathbf{f}_* \otimes 1)(p \circ \mathbf{f})$.

Next, if $g$ is a linear combination of simple tensors, one can check directly that $\norm{g \circ \mathbf{f}}_{\mathbf{R},2} \leq \norm{g}_{\mathbf{R'}
,2}$ by considering evaluations on all possible tuples.  This estimate allows us to pass to the completion, showing that if $g \in \mathcal{F}_{\mathbf{R}',2}$, then $g \circ \mathbf{f} \in \mathcal{F}_{\mathbf{R},2}$.

Again, by evaluating on points, we see that $\norm{g \circ \mathbf{f}}_{\mathbf{R},\infty} \leq \norm{g}_{\mathbf{R}',\infty}$.  Replacing the single function $g$ by an $I''$-tuple yields the asserted result.
\end{proof}

The second continuity property that we need for $\mathcal{F}_{\mathbf{R},\mathbf{R}'}$ is uniform continuity with respect to $L^2$ norm.  This will allow us to use the functions in $\mathcal{F}_{\mathbf{R},\mathbf{R}'}$ to ``push forward'' $\eps$-coverings from one microstate space to another.

\begin{prop}[{\cite[Proposition 3.9]{FreePinsker}}] \label{prop:L2uniformcontinuity}
Let $I$ be an index set and $\mathbf{R} \in (0,+\infty)^I$, and let $f \in \mathcal{F}_{\mathbf{R},2}$.  Then for every $\eps > 0$ there exists a finite $F \subseteq I$ and a $\delta > 0$ such that for every $(\cM,\tau)$ and $\mathbf{x}$, $\mathbf{y} \in \cM_{sa}^I$ with $\norm{x_i}, \norm{y_i} \leq R_i$, if $\norm{x_i - y_i}_2 < \delta$ for all $i \in F$, then $\norm{f(\mathbf{x}) - f(\mathbf{y})}_2 < \eps$.
\end{prop}

\subsection{Proof of monotonicity and invariance}

\begin{thm}
Let $(M,\tau)$ be a tracial von Neumann algebra.  Let $I$ and $J$ be index sets, let $\mathbf{R} \in (0,\infty)^I$ and $\mathbf{S} \in (0,\infty)^J$, and let $\mathbf{x} \in M_{s.a.}^I$ and $\mathbf{y} \in M_{s.a.}^J$ with $\norm{x_i} \leq R_i$ for $i \in I$ and $\norm{y_j} \leq S_j$ for $j \in J$.  Let $I'$, $J'$, $\mathbf{R}'$, $\mathbf{S}'$, $\mathbf{x}'\in M^{I'}_{s.a.}$, $\mathbf{y}'\in M^{J'}_{s.a.}$ satisfy similar conditions.  Suppose that $\mathrm{W}^*(\mathbf{x},\mathbf{y}) \supseteq \mathrm{W}^*(\mathbf{x}',\mathbf{y}')$ and $\mathrm{W}^*(\mathbf{x}) \subseteq \mathrm{W}^*(\mathbf{x}')$.  Then
\[
h_{\mathbf{R},\mathbf{S}}(\mathbf{x}:\mathbf{y}) \leq h_{\mathbf{R}',\mathbf{S}'}(\mathbf{x}':\mathbf{y}').
\]
\end{thm}

\begin{proof}
Unwinding the suprema and infima in the definition of $h$, it suffices to show that for every $F \subseteq I$ finite and $\eps > 0$, there exists $F' \subseteq I'$ finite and $\eps' > 0$, such that for every neighborhood $\cO'$ of $\lambda_{(\mathbf{x}',\mathbf{y}')}$ in $\Sigma_{(\mathbf{R}',\mathbf{S}')}$, there exists a neighborhood $\cO$ of $\lambda_{(\mathbf{x},\mathbf{y})}$ in $\Sigma_{(\mathbf{R},\mathbf{S})}$ such that
\[
K_{F,\eps}^{\orb}(\pi_I(\Gamma_{(\mathbf{R},\mathbf{S})}^{(n)}(\cO))) \leq K_{F',\eps'}^{\orb}(\pi_{I'}(\Gamma_{(\mathbf{R}',\mathbf{S}')}^{(n)}(\cO'))).
\]

Fix $F$ and $\eps$.  Because $\mathrm{W}^*(\mathbf{x}) \subseteq \mathrm{W}^*(\mathbf{x}')$, there exists $\mathbf{f} \in \mathcal{F}_{\mathbf{R}',\mathbf{R}}$ such that $\mathbf{x} = \mathbf{f}(\mathbf{x}')$.  By Proposition \ref{prop:L2uniformcontinuity}, there exists $\eps' > 0$ and $F' > 0$ such that for all tracial von Neumann algebras $N$ and all $\mathbf{z}$, $\mathbf{w} \in N^{I'}$ with $\norm{z_i}, \norm{w_i} \leq R_i'$, we have $\max_{i \in F'} \norm{z_i - w_i}_2 < 2\eps'$ implies $\max_{i \in F} \norm{f_i(\mathbf{z}) - f_i(\mathbf{w})}_2 < \eps/2$.

Now fix a neighborhood $\cO'$ of $\lambda_{\mathbf{x}',\mathbf{y}'}$ in $\Sigma_{(\mathbf{R}',\mathbf{S}')}$.  We claim that
\[
K_{F,\eps/2}^{\orb}(\mathbf{f}(\pi_{I'}(\Gamma_{\mathbf{R}',\mathbf{S}'}^{(n)}(\cO')))) \leq K_{F',\eps'}^{\orb}(\pi_{I'}(\Gamma_{\mathbf{R}',\mathbf{S}'}^{(n)}(\cO')))
\]
Indeed, let $\Omega$ be a set of cardinality $K_{F',\eps'}^{\orb}(\pi_I(\Gamma_{\mathbf{R},\mathbf{S}}^{(n)}(\cO')))$ such that $\pi_I(\Gamma_{\mathbf{R},\mathbf{S}}^{(n)}(\cO')) \subseteq N_{F',\eps'}^{\orb}(\Omega)$.  Let $\Omega' \subseteq \pi_I(\Gamma_{\mathbf{R},\mathbf{S}}^{(n)}(\cO'))$ be chosen to have one element within $\eps'$ of each element of $\Omega$, so that $\pi_I(\Gamma_{\mathbf{R},\mathbf{S}}^{(n)}(\cO')) \subseteq N_{F',2\eps'}^{\orb}(\Omega')$.  Then each $\mathbf{X}' \in \Omega'$, and more generally in $\pi_I(\Gamma_{\mathbf{R},\mathbf{S}}^{(n)}(\cO'))$ satisfies $\norm{X_i'} \leq R_i$, so that it is valid to apply the uniform continuity estimate for $\mathbf{f}$ to such points $\mathbf{X}'$.  The choice of $(F,\eps)$ thus implies that
\[
\mathbf{f}(\pi_{I'}(\Gamma_{\mathbf{R}',\mathbf{S}'}^{(n)}(\cO'))) \subseteq N_{F',\eps'/2}^{\orb}(\mathbf{f}(\Omega')),
\]
which proves our claim about the covering numbers.

Next, we describe how to choose $\cO$.  Since $\mathrm{W}^*(\mathbf{x}',\mathbf{y}') \subseteq \mathrm{W}^*(\mathbf{x},\mathbf{y})$, there exists $\mathbf{g} \in \mathcal{F}_{(\mathbf{R},\mathbf{S}),(\mathbf{R}',\mathbf{S}'))}$ such that $(\mathbf{x}',\mathbf{y}') = \mathbf{g}(\mathbf{x},\mathbf{y})$.  By continuity of $\mathbf{g}_*: \Sigma_{(\mathbf{R},\mathbf{S})} \to \Sigma_{(\mathbf{R}',\mathbf{S}')}$, the set
\[
\cO_1 = (\mathbf{g}_*)^{-1}(\cO')
\]
is open.   Let
\[
\cO_2 = \{\lambda_{(\mathbf{z},\mathbf{w})} \in \Sigma_{(\mathbf{R},\mathbf{S})}: \max_{i \in F} \norm{f_i \circ \pi_{I'} \circ \mathbf{g}(\mathbf{z},\mathbf{w}) - z_i}_2 < \eps/2 \}.
\]
The set $\cO_2$ is open using Propositions \ref{prop: composition} and \ref{prop:pushforwardcontinuity}.  It also contains $\lambda_{(\mathbf{x},\mathbf{y})}$ because $\mathbf{f}(\pi_{I'}(\mathbf{g}(\mathbf{x},\mathbf{y}))) = \mathbf{f}(\pi_{I'}(\mathbf{x}',\mathbf{y}')) = \mathbf{f}(\mathbf{x}') = \mathbf{x}$.

Let $\cO = \cO_1 \cap \cO_2$.  We claim that
\[
\pi_I \Gamma_{\mathbf{R},\mathbf{S}}^{(n)}(\cO) \subseteq N_{F,\eps/2}^{\orb}(\mathbf{f}(\pi_{I'}(\Gamma_{\mathbf{R}',\mathbf{S}'}^{(n)}(\cO')))).
\]
Indeed, if $\mathbf{X}$ is in the set on the left-hand side, then there exists $\mathbf{Y}$ such that $\lambda_{(\mathbf{X},\mathbf{Y})} \in \cO$.  In particular, this means that $\lambda_{\mathbf{g}(\mathbf{X},\mathbf{Y})} \in \cO'$, or in other words $\mathbf{g}(\mathbf{X},\mathbf{Y}) \in \Gamma_{\mathbf{R}',\mathbf{S}'}^{(n)}(\cO')$.  Moreover, $\max_{i \in F} \norm{X_i - f_i \circ \pi_{I'} \circ \mathbf{g}(\mathbf{X},\mathbf{Y})}_2 < \eps / 2$.  Therefore, $\mathbf{X}$ is in the $(F,\eps/2)$-neighborhood of $\mathbf{f} \circ \pi_{I'}$ of some point in $\Gamma_{\mathbf{R}',\mathbf{S}'}^{(n)}(\cO')$, which proves the claimed inclusion.

This inclusion $\pi_I \Gamma_{\mathbf{R},\mathbf{S}}^{(n)}(\cO) \subseteq N_{F,\eps/2}^{\orb}(\mathbf{f}(\pi_{I'}(\Gamma_{\mathbf{R}',\mathbf{S}'}^{(n)}(\cO'))))$ in turn implies that
\begin{align*}
K_{F,\eps}^{\orb}(\pi_I \Gamma_{\mathbf{R},\mathbf{S}}^{(n)}(\cO)) &\leq K_{F,\eps/2}^{\orb}(\mathbf{f}(\pi_{I'}(\Gamma_{\mathbf{R}',\mathbf{S}'}^{(n)}(\cO')))) \\
&\leq K_{F',\eps'}^{\orb}(\pi_{I'}(\Gamma_{(\mathbf{R}',\mathbf{S}')}^{(n)}(\cO'))),
\end{align*}
where the second inequality is the earlier claim that we proved.
\end{proof}

This theorem implies the following:
\begin{itemize}
    \item In the case where $\mathbf{x} = \mathbf{x}'$ and $\mathbf{y} = \mathbf{y}'$, the theorem shows that $h_{\mathbf{S},\mathbf{R}}(\mathbf{x}:\mathbf{y})$ is independent of $\mathbf{R}$ and $\mathbf{S}$ so long as $\norm{x_i} \leq R_i$ for $i \in I$ and $\norm{y_j} \leq S_j$ for $j \in J$.  Thus, we may unambiguously write $h(\mathbf{x}:\mathbf{y})$.
    \item In the case where $\mathrm{W}^*(\mathbf{x},\mathbf{y}) = \mathrm{W}^*(\mathbf{x}',\mathbf{y}')$ and $\mathrm{W}^*(\mathbf{x}) = \mathrm{W}^*(\mathbf{x}')$, the theorem shows that $h(\mathbf{x}:\mathbf{y}) = h(\mathbf{x}':\mathbf{y}')$.  Hence for $N \leq M$, we may unambiguously define $h(N:M)$ as $h(\mathbf{x}:\mathbf{y})$ for some tuples $\mathbf{x}$ and $\mathbf{y}$ such that $N = \mathrm{W}^*(\mathbf{x})$ and $M = \mathrm{W}^*(\mathbf{x},\mathbf{y})$.
    \item Now suppose that $P \leq N \leq M$.  Applying the theorem in the case where $\mathrm{W}^*(\mathbf{x},\mathbf{y}) = \mathrm{W}^*(\mathbf{x}',\mathbf{y}') = M$ and $P = \mathrm{W}^*(\mathbf{x}) \subseteq \mathrm{W}^*(\mathbf{x}') = N$, we obtain $h(P:M) \leq h(N:M)$.
    \item Again, suppose $P \leq N \leq M$. Applying the theorem in the case where $\mathrm{W}^*(\mathbf{x},\mathbf{y}) = M \supseteq N = \mathrm{W}^*(\mathbf{x}',\mathbf{y}')$ and $P = \mathrm{W}^*(\mathbf{x}) = \mathrm{W}^*(\mathbf{x}')$, we obtain $h(P:M) \leq h(P:N)$.
\end{itemize}


\end{document}